\documentclass{article}
\pdfoutput=1
\usepackage{arxiv}
\usepackage{color}
\usepackage{pdfpages}
\usepackage{hyperref}
\usepackage{amsmath, amsfonts, amsthm, amssymb}
\usepackage{verbatim}
\usepackage{stmaryrd}
\usepackage{fancyhdr}
\usepackage{color}
\usepackage{subfigure}
\usepackage{graphicx}
\usepackage{fullpage,epsf}
\usepackage[english]{babel}
\usepackage{enumitem}
\usepackage{float}
\usepackage{listings}
\usepackage[toc,page]{appendix}
\usepackage{kantlipsum}
\allowdisplaybreaks
\usepackage[utf8]{inputenc}
\font\twelvemsb=msbm10 at 12pt
\newfam\msbfam
\textfont\msbfam=\twelvemsb

\usepackage{mathtools}

\newtheorem{theorem}{Theorem}[section]

\newtheorem{corollary}[theorem]{Corollary}
\newtheorem{lemma}[theorem]{Lemma}
\newtheorem{proposition}[theorem]{Proposition}

\theoremstyle{definition}

\newtheorem{remark}[theorem]{Remark}

\newtheorem{assumption}[theorem]{Assumption}

\newtheorem{exmp}{Example}[section]
\newcommand{\LL}{\mathcal{L}}

\newcommand{\V}{\mathcal{V}}

\newcommand{\RR}{\mathbb{R}}

\newcommand{\NN}{\mathbb{N}}
\newcommand{\EE}{\mathbb{E}}

\usepackage{cite}

\begin{document}

\title{Decaying derivative estimates for functions of solutions to non-autonomous SDEs}

\author{\textbf{Maria Lefter}\thanks{School of Mathematics, University of Edinburgh}\\ \href{mailto:s1776026@ed.ac.ukm}{s1776026@ed.ac.uk} 
\and \textbf{ David \v{S}i\v{s}ka}\footnotemark[1] \\   \href{mailto:d.siska@ed.ac.uk}{d.siska@ed.ac.uk}
\and  \textbf{\L{}ukasz Szpruch}\footnotemark[1] \,\,\thanks{Alan Turing Institute} \\  \href{mailto:l.szpruch@ed.ac.uk}{l.szpruch@ed.ac.uk}}
\maketitle

\begin{abstract}

We produce uniform and decaying bounds in time for derivatives of the solution to the  backwards Kolmogorov equation associated to a stochastic processes governed by a time dependent dynamics. These hold under assumptions over the integrability properties in finite time of the derivatives of the transition density associated to the process, together with the assumption of remaining close over all $[0,\infty)$, or decaying in time, to some static measure. We moreover provide examples which satisfy such a set of assumptions. Finally, the results are interpreted in the McKean--Vlasov context for monotonic coefficients by introducing an auxiliary non-autonomous stochastic process.  

\end{abstract}

\section{Introduction}

In this paper we consider the following real--valued, $d$-dimensional stochastic process $(X_{t})_{t\ge 0}$ satisfying a non--autonomous SDE. Indeed, given time--dependent coefficients $b:[0,\infty)\times \RR^d  \to \RR^d $ and $\sigma:[0,\infty)\times \RR^d  \to \RR^d \times \RR^d$ and initial datum $x\in \RR^d$,  our object of study is the stochastic process $(X_s^{x})_{s\ge 0}$ assumed to be the unique (in the sense of probability law) weak solution of the following SDE with Brownian motion process $(B_s)_{s\ge0}$:
\begin{equation}\label{original_intro}
dX_s^{x}=b(s,X_s^{x})ds +\sigma(s,X_s^{x})dB_s,\quad \forall s \in(0,\infty); \quad X_0^{x}=x.
\end{equation}
Let $\phi:\RR^d \mapsto \RR$ be a measurable function. Then, due to uniqueness of solution to \eqref{original_intro},  $[0,\infty)\times \RR^d\ni (s,x)\mapsto V(s,x):=\EE[\phi(X_s^{0,x})]$ is a well defined function under mild conditions on $\phi$ and the coefficients $b$ and  $\sigma$. 
Moreover it satisfies, under enough assumptions for the coefficients, a certain PDE known as the backward Kolmogorov equation  (see e.g.~\cite[Theorem 7.6]{karatzas1998brownian} or \cite{krylov1999kolmogorov}).

For only space dependent, smooth coefficients $b, \sigma$ with bounded derivatives of any order, in addition of $\sigma$ being bounded itself; and smooth function $\phi$, the authors of  \cite{talay1990second} were motivated by uniform weak error for  Euler scheme estimates to obtain exponential decay for derivatives of $V$ of any order. For strictly space dependent,  bounded and H\"older continuous coefficients with derivatives up to certain order, we have the results of \cite{pardoux2003poisson} which are heavily based on previously obtained rate of convergence to invariant measures in \cite{veretennikov1997polynomial}. More recent work was done by Menozi, Pesce and Zhang \cite{menozzi2020density} for uniform bounds for $DV(s,x)$ and $D^2V(s,x)$, where the coefficients are assumed  H\"older continuous in space and the drift has linear growth.

The novelty of our results is that we translate finite time results to infinite time. More precisely, 
we obtain explicit estimates in the time interval $[0,\infty)$ of the space derivatives of $V$ up to the order imposed by the regularity and integrability  of the transition density of the SDE (\ref{original_intro}) in a fixed finite time interval. The order of decay is given by what looks like convergence to an invariant measure but is in fact much weaker (see Assumption \ref{as3}).


An application is the use of such derivative estimates to obtain explicit weak error rates  when approximating a process $(X^x_s)_{s\ge0}$ by an Euler scheme. Namely, the definition of $V$ and the initial data of the PDE \eqref{original_intro}, allows us to recast the expression of the weak error into another one to which we can apply It\^o's formula. Additionally, the use of the backward Kolmogorov equation simplifies the expression result of this computation and lets us split the weak error into more approachable terms. Together with decaying in time bounds on the space derivatives of $V$, these allow obtaining uniform such weak error orders. In the case that $(X^x_s)_{s\ge0}$ follows a McKean--Vlasov dynamics (for more information see Section \ref{ap_mc_sec}), the decaying in time estimates on the space derivatives of $V$ allow obtaining explicit, uniform in time weak error order for approximating particle systems (see \cite{chassagneux2019weak,bencheikh2019weak}). 

The paper is organised as follows: in Section \ref{sec main} we formulate conditions under which we prove that the desired derivative estimates hold. They are presented in the first main result of the paper: Theorem \ref{main}. Its proof is inspired by \cite{pardoux2003poisson} in the sense that we also use Chapman--Kolmogorov identity which allows us to move the derivatives of $V$ onto the derivatives of the transition density of $X$ in finite time intervals.   
In Section 3 we then show two examples satisfying this set of assumptions in the non-autonomous SDE setup and we state the other main result, which brings down to earth the results in Section \ref{sec main} and is presented in Theorem \ref{WTV}. Finally, in Section \ref{ap_mc_sec} we interpret the result in the McKean--Vlasov SDE scenario.


\subsection{Notation}
Relevant spaces for this paper are  the space of $\RR^d$--Borel--measurable functions $B(\RR^d)$ and, for $m\in \NN$, $B_m(\RR^d)=\big\{\phi\in B(\RR^d)$ such that there exists $C>0$ and $|\phi(x)|\le C(1+|x|^m),\, \forall x\in \RR^d\big\}$. Moreover, we consider the space of $p$-continuously differentiable ($p\in \NN$) functions  on $\RR^d$ represented by $C^p(\RR^d)$; and for $p_1,p_2\ge1$,  $I\subseteq \RR_+ \equiv [0,\infty)$,  the space of $p_1,p_2$-continuously differentiable functions over $I$ and $\RR^d$ respectively denoted by  $C^{p_1,p_2}(I\times\RR^d)$. A $b$ as the subindex, i.e $C_b^{p_1,p_2}([0,T]\times\RR^d)$,  will mean that the functions and all the required derivatives are bounded.
A particularly important role is played by the space of probability measures on $\RR^d$ represented by $\mathcal{P}(\RR^d)$.  Also, we say $\mathcal{L}(\xi) \in \mathcal{P}^W(\RR^d)$ if for given $W$ a Lyapunov function (in the sense of Assumption \ref{asex3c}) we have  $\int_{\RR^d} (1+W(x))\mathcal{L}(\xi)(dx)<\infty$. On the other hand, for any function $\phi\in B(\RR^d)$, we say $\phi \in L^p(q)$ when   $\int_{\RR^d} |\phi(x)|^pq(dx)<\infty.$ The measure is only omitted when we refer to the Lebesgue measure.

With respect to distances, we denote the $m$--Wasserstein measure by  $\mathcal{W}_m$ and the Total Variation norm by $||\cdot||_{TV}$ (see \cite[pp. 436, 244]{bakry2014analysis}). Additionally, recall \cite{villani2008optimal} that for a given function $W$ taking values in $[0,\infty)$, the Weighted Total Variation (WTV) norm, $||\cdot||_W$, and WTV distance, $d_W$, are given for any two measures $q,q'\in \mathcal{P}^W(\RR^d)$ by 
\begin{equation}\label{WTV_def}
\begin{split}
||q||_W:=&\int_{\RR^d}(1+W(x))dq(x)\\
d_W(q,q'):=&\int_{\RR^d}(1+W(x))d|q-q'|(x)=2||(1+W(\cdot))(q-q')||_{TV}=||q-q'||_W.
\end{split}
\end{equation}
Moreover, in any Euclidean space  $|\cdot|$ is used as an arbitrary norm (since they are all equivalent in finite dimensions  and our estimates would only change up to a constant), only the trace norm denoted by $\hbox{tr}(\cdot)$ being mentioned independently. Moreover, for all $x,y\in \RR^d$ we use $xy$ to denote the dot product. Finally, for any matrix $M\in \RR^{d\times d}$, we denote its transpose by $M^*.$

Next, let us mention that throughout the paper $C>0$ is a constant changing value from line to line, which might be dependent on the dimension, the coefficients and the remaining parameters in the set of assumptions but crucially independent of time and space. 

Finally, let us mention that we do not distinguish in between measures and densities, i.e if $p \in \mathcal{P}(\RR^d)$ is absolutely continuous with respect to the Lebesgue measure on $\RR^d$, we also denote it's density by $p$. Moreover, if they exist, we denote the $n^{th}$-derivatives in the first space component of the density $(0,\infty)\times \RR^d\ni (s,x') \mapsto p_{s}^{0}(x,x')$  by $\partial_x^\alpha p_{s}^{0}(x,\cdot)=\partial_{x_1}^{\alpha_1}...\partial_{x_d}^{\alpha_d}p_{s}^{0}(x,\cdot)$, where $x=(x_i)_{i=1,...,d}\in \mathbb R^d$ and $\alpha=(\alpha_i)_{i=1,...,d}\in \NN^d$ is such that $|\alpha|=n$. As an abuse of notation, when one does not care about the specific derivative but only about its order, we write $\partial_x^\alpha=\partial_x^n$.

\section{Main result: Derivative with respect to the initial condition}
\label{sec main}
For $x\in \mathbb R^d$ and $\tau \geq 0$, we consider the following non-autonomous SDE 
\begin{equation}\label{original}
X_{\tau+s}^{x,\tau}=x+\int_{\tau}^{\tau+s}b(t, X_t^{x,\tau})\,dt +\int_{\tau}^{\tau+s}\sigma(t, X_{t}^{x,\tau})\,dB_t,\quad s\ge0. 
\end{equation}
In order to avoid cumbersome notation,  we will be using the following convention $X_{0+s}^{x,0}=X_{s}^{x}$ for all $(s,x)\in [0,\infty)\times\mathbb R^d.$

To formulate the assumptions let us fix  $\mathcal S\subseteq B(\RR^d)$ and $N\in \NN.$

\begin{assumption}[Conditions for the density of the transition probability]\label{as4}
For all $x\in \RR^d,\tau\ge0$, the equation \eqref{original} has a solution  $\big(\Omega, \mathbb{P},(B_s)_{s\ge\tau}, (\mathcal{F}_s)_{s\ge\tau},(X_s^{x,\tau})_{s\ge\tau}\big)$  unique in the sense of probability. 

Moreover, we suppose that this process admits a density: $(s,x') \mapsto p_{\tau+s}^{\tau}(x,x')$. 

Finally, either $\mathcal{S}\subseteq  C(\RR^d)$ or for any $0\leq n \leq N$  there exists $\delta>0$ such that 
\[
\underset{\phi\in \mathcal{S}}{\sup}\int_{\RR^d}|\phi(x'')\partial_x^np_{s}^{0}(x,x'')|^{1+\delta}dx''<\infty \,,\,\,\, \forall (s,x) \in (0,\infty)\times\RR^d.
\]
\end{assumption}

\begin{assumption}[Sticking to a measure]
\label{as3}
There exist $q\in \mathcal{P}(\mathbb R^d) $, $g: \mathbb R^d \to \mathbb R_+$ and $G: \mathbb R_+ \to \mathbb R_+$ such that for all $(s,x)\in (0,\infty)\times\mathbb R^d$:
\[\underset{\phi\in \mathcal{S}}{\sup}\,
\bigg|\int_{\RR^d} \phi(x')\big( p_{1+s}^{1}(x,dx')-q(dx')\big)\bigg|\le g(x)G(s).
\]	
\end{assumption}

Assumption \ref{as3} deserves a few comments. First, in most scenarios we would like $\lim_{s\to \infty} G(s) = 0$, i.e the transition density is not only sticking to a static measure but decaying to one. 
Next, by considering a tailor--made family of test functions, one can make use of familiar metrics in order to verify Assumption \ref{as3}. Indeed, suppose that $\mathcal S$ is the family of bounded measurable functions, then it is enough for the law of the the solution to \eqref{original} to converge to $q$ in the Total Variation distance for the above assumption to be satisfied. 
Now, if $\mathcal S$ is formed only by  1--Lipschitz functions then it is enough for $q$ to be a limit in the 1--Wasserstein distance of the law of the process satisfying \eqref{original}.
Yet another possibility is to consider $\mathcal S$ as the family of locally Lipschitz functions and have convergence of the solution to \eqref{original} to $q$ in the corresponding Wasserstein distance (see Lemma \ref{WTV}). 

A key remark is that $q$ is just a limiting measure (if $G$ decreases to $0$), and with Assumption \ref{as3} we are not covertly asking for the existence of a unique invariant measure in any of the mentioned distances. Moreover, if  $G$ in not a decreasing function but a merely bounded one, this assumption requires the law of the the solution to \eqref{original} to stay ``close'' to the fixed measure $q$, without having to converge to it. Consequently, the space derivatives of the functions $V$ can only be concluded uniform but not decaying in time when applying Theorem \ref{main}.

\begin{assumption}[Smoothness and integrability of derivatives of the density w.r.t. the starting point]\label{as2} For any $0\leq n \leq N$, we assume $(s,x',x) \mapsto \partial_x^np_{s}^{0}(x,x')$ exist and are continuous in $(x,x')$. Moreover, there exists  $h:\RR^d \to \RR_{+}$ such that for any $1\leq n \leq N$ and $g$ satisfying Assumption \ref{as3}, the following is satisfied:
\begin{align*}
\int_{\RR^d}g(x'')&|\partial_x^np_{1}^{0}(x,x'')|dx''\le h(x)\,,\,\,\, \forall x\in \mathbb R^d.
\end{align*}

\end{assumption}
\begin{theorem}\label{main}
Let $\phi \in \mathcal S$, $(X_s^x)_{s\ge0}$ be the unique (in law) solution of \eqref{original} and 
\[
V(s,x):=\int_{\RR^d}\phi(x')p_s^0(x,x')\,dx'=\EE[\phi(X_s^x)]\,.
\]	
If Assumptions \ref{as4}, \ref{as3} and \ref{as2} hold then, 
for all $1\le n\le N$, we have
\[
|\partial^n_xV(s,x)|\le h(x)G(s),\quad for\, all \quad (s,x)\in(1,\infty)\times \RR^d.
\]
\end{theorem}
Note that, if $\lim_{s\to \infty} G(s) = 0$ we conclude that the derivatives w.r.t. $x$ of $V(s,x)$ decay to zero.

\begin{proof}

Let us first show  that due to Assumption \ref{as4},   for all $1\le n\le N$ and $\forall x\in \RR^d$
\begin{equation}\label{step1}
\partial^n_xV(s,x)=\partial_x^n\int_{\RR^d}\phi(x'')p_{s}^0(x,x'')dx''=\int_{\RR^d}\phi(x'')\partial_x^np_{s}^0(x,x'')dx''.
\end{equation}
Indeed, if $\phi$ is continuous, we can differentiate under the integral sign with Leibniz' formula (see \cite[Theorem 12.14]{bartle2001modern}). Otherwise we argue as follows.
Recall that for all $1\le n\le N$ and $\forall x\in \RR^d$ we assumed the existence of a $\delta>0$ such that  $\int_{\RR^d}|\phi(x'')|^{1+\delta}|\partial_x^np_s^0(x,x'')|^{1+\delta}dx''<\infty$. As a consequence, for any order $1\le n\le N$ and  $h>0, \,(s,x'') \in (0,\infty)\times\RR^d$ and any element in an orthonormal basis in $\RR^d$ represented as $\{e_i\}_{1,...,d}$,
\begin{align*}
\underset{h\ge0}{\sup }\int_{\RR^d}&\Big|
\phi(x'')\frac{1}{h}\Big(\partial_x^{n-1}p_{s}^0(x+he_i,x'')-\partial_x^{n-1}p_{s}^0(x,x'')\Big)\Big|^{1+\delta}dx''\\
&\le \int_{\RR^d}|
\phi(x'')|^{1+\delta}\underset{h\ge 0}{\sup }\,\Big|\frac{1}{h}\Big(\partial_x^{n-1}p_{s}^0(x+he_i,x'')-\partial_x^{n-1}p_{s}^0(x,x'')\Big)\Big|^{1+\delta}dx''
\\
&\le \int_{\RR^d}|
\phi(x'')|^{1+\delta}|\partial_x^{n}p_{s}^0(x,x'')|^{1+\delta}dx''<\infty.
\end{align*}
Meaning that by  De La Vall\'ee Poussin Theorem (see \cite[Theorem 2.4.4]{attouch2014variational}), for any order $1\le n\le N$ and $\forall x\in \RR^d$, the increments 
\[
\phi(x'')\frac{1}{h}\Big(\partial_x^{n-1}p_{s}^0(x+he_i,x'')-\partial_x^{n-1}p_{s}^0(x,x'')\Big), \quad i=1,...,d;
\]
are uniformly (in $h$) integrable in $x''$ over $\RR^d$ . This means that we can apply Vitali's Convergence Theorem (see \cite[Chapter 4]{brezis2010functional}) and obtain \eqref{step1} by induction from:
\begin{align*}
\underset{h \to 0}{\lim }\int_{\RR^d}\phi(x'')\frac{\partial_x^{n-1}p_{s}^0(x+he_i,x'')-\partial_x^{n-1}p_{s}^0(x,x'')}{h}dx''&= \int_{\RR^d}\phi(x'')\,\underset{h \to 0}{\lim }\,\frac{\partial_x^{n-1}p_{s}^0(x+he_i,x'')-\partial_x^{n-1}p_{s}^0(x,x'')}{h}dx''.
\end{align*}

Let $s\ge 1$ and recall the non-autonomous Chapman--Kolmogorov identity 
\begin{equation}\label{ck}
p_{\tau+s}^{\tau}(x,x'')=\int_{\RR^d} p_{1+\tau}^{\tau}(x,x')p_{\tau+s}^{1+\tau}(x',x'')dx',
\end{equation}
whose proof we include in Appendix \ref{CK}. Let us now apply it with $\tau=0$ in \eqref{step1} for any  $1\le n\le N$ and $x\in \RR^d$. After taking the derivatives inside the second integral with Leibniz' formula (this is allowed given the continuity of $(x,x')\mapsto \partial_x^np_s^0(x,x')$ stated in Assumption \ref{as3} for $0\le n\le N$, see \cite[Theorem 12.14]{bartle2001modern}), we obtain:
\begin{align*}
|\partial^n_xV(s,x)|&=\bigg|\int_{\RR^d}\phi(x'')\partial_x^np_{s}^0(x,x'')dx''\bigg|=\bigg|\int_{\RR^{d}}\phi(x'')\partial_x^n\Big(\int_{\RR^{d}}p_{1}^{0}(x,x')p_{1+s}^{1}(x',x'')dx'\Big)dx''\bigg|\\
&=\bigg|\int_{\RR^{d}}\int_{\RR^{d}}\phi(x'')\partial_x^np_{1}^{0}(x,x')p_{1+s}^{1}(x',x'')dx'dx''\bigg|.
\end{align*}

Now notice that first by Fubini's Theorem, afterwards by Leibniz' formula due to continuity of $\partial_{x}^np_{1}^{0}(x,x')$ for $0\le n\le N$ and finally by  the fact that $q$ is independent of the initial data, we conclude:
\begin{align}\label{normalization}
\nonumber\int_{\RR^d}\int_{\RR^d}\phi(x'')\partial_{x}^np_{1}^{0}(x,x')q(dx'')dx'&=\int_{\RR^d}\phi(x'')\int_{\RR^d}\partial_{x}^np_{1}^{0}(x,x')dx'q(dx'')\\\nonumber
&=\int_{\RR^d}\phi(x'')\partial_{x}^n\bigg(\int_{\RR^d}p_{1}^{0}(x,x')dx'\bigg)q(dx'')\\
&=\int_{\RR^d}\phi(x'')\partial_{x}^n(1)q(dx'')\\\nonumber
&=0.
\end{align}
This allows us to continue the above chain of equalities as
\begin{align*}
|\partial_x^nV(s,x)|&=\bigg|\int_{\RR^{d}}\int_{\RR^{d}}\phi(x'')\partial_x^np_{0}^{1}(x,x')\Big(p_{1+s}^{1}(x',dx'')-q(dx'')\Big)dx'\bigg|\\
&\le \int_{\RR^d}|\partial_x^np_{1}^{0}(x,x')|\bigg|\int_{\RR^{d}}\phi(x'')\Big(p_{1+s}^{1}(x',dx'')-q(dx'')\Big)\bigg|dx'.
\end{align*}
Since $\phi \in \mathcal{S}$ and due to Assumption~\ref{as3} together with Assumption~\ref{as2}, we conclude that
\begin{align*}
|\partial_x^nV(s,x)|&\le  \int_{\RR^d}|\partial_x^np_{1}^{0}(x,x')|g(x') G(s)dx'\\
&\le  h(x)G(s).
\end{align*}
This completes the proof.
\end{proof}

To finish this section, let us enumerate a few  well studied possibilities of obtaining Assumptions \ref{as4}, \ref{as3} and \ref{as2}, which guarantee the estimates in Theorem \ref{main}.

First, one can extract derivative bounds of the transition density of $(X_s^{\tau;\xi})_{s\in [\tau,T]}$  in finite time from the PDEs literature.  Many results are available in this direction when  we allow the time interval to be of a fixed length $T-\tau>0$: Friedman in \cite{friedman2008partial} and Eidelman in \cite{eidelman2012parabolic} are two main references. Their restrictions come from the smoothness required for the coefficients: bounded and uniformly H\"older continuous for Friedman and bounded diffusion and linearly growing drift for Eidelman. An extension of the later was obtained recently by Menozi, Pesce and Zhang in \cite{menozzi2020density}. In this later result the diffusion is assumed to be H\"older continuous in space and the drift to have linear growth. For a general study of transition densities, see \cite{bogachev2015fokker}. 

Second, one can classify the methods for obtaining decay to the invariant measure, if what we are after are decaying derivates bounds, in three categories. For each, we include a single reference and the reader is remitted to \cite{bogachev2019convergence} for a more detailed list of references for: (i) the approach based on Harris theorem or the Meyn--Tweedie approach with Lyapunov functions \cite{hairer2011yet};(ii) the approach based on entropy estimates and Poincar\'e and Sobolev inequalities \cite{arnold2001convex}; and (iii) the probabilistic approach based on coupling \cite{eberle2017quantitative}.

\section{Application to non--autonomous SDEs}\label{non_autonomous}

This section is dedicated to finding a tractable set of assumptions which imply in turn Assumptions \ref{as3} and \ref{as2}. Provided the road map in the previous section, we show two alternatives for obtaining each of the mentioned assumptions, any of which when combined imply that the  conclusion to Theorem \ref{main} is valid for the associated family of test functions. In particular, we show four examples (although two are a particular case of the other ones, see Section \ref{partic_case}) for which we conclude exponential decay of the first two derivates of $V(t,x)=\EE[\phi(X_t^{x})]$. Such results are presented in Theorems \ref{thex1} and \ref{thex3}.

\begin{assumption}[Regularity and growth of the coefficients]\label{asex2} Assume one of the following regularity and growth conditions holds for the coefficients:

\begin{enumerate}

\item[a]\label{asex2a} The coefficients $b, \sigma \in C_b^{1,2}([0,1]\times \RR^d)$  and $b, \sigma \in C^{1,2}([0,\infty)\times \RR^d)$. We denote by $M>0$ the bound on the diffusion.

\item[b]\label{asex2b} 
The diffusion $\sigma\in C^{0,2}([0,\infty)\times\RR^d)$. Moreover, $\sigma$ is continuous in $t$ uniformly  in $x$. For the drift we assume  $b \in C^{0,2}([0,\infty)\times\RR^d)$. 

Additionally we assume that  there exist $M_0\ge0$ and $0<\lambda<1$ such that for all $t\in[0,1]$ and $x,y\in \RR^d$, $$|b(t,x)-b(t,y)|,|\sigma(t,x)-\sigma(t,y)|\le M_0|x-y|^{\lambda};$$ 
and there exist  $M,\epsilon>0$  such that $\forall( t,x) \in [0,1]\times \RR^d$,
\[
|b(t,x)|\le M(1+ |x|), \quad |\partial_xb(t,x)|\le M(1+ |x|)^{1-\epsilon}, \quad |\partial_x^2b(t,x)|\le M(1+ |x|)^{3-\epsilon};
\]
and
\[\quad|\sigma(t,x)|\le M, \quad\quad\quad\quad |\partial_x\sigma(t,x)|\le M(1+ |x|)^{1-\epsilon}, \quad |\partial_x^2\sigma(t,x)|\le M(1+ |x|)^{2-\epsilon}.
\]

\end{enumerate}
\end{assumption}

\begin{assumption}[Uniform ellipticity]\label{asex20}
Let us consider the differential operators $L,L^*$ defined by 
$$C^2(\RR^d)\ni u \mapsto Lu=\frac{1}{2} \text{tr}\big(\sigma\sigma^*\partial_x^2u\big)+b\partial_xu\qquad \hbox{and} \qquad C^2(\RR^d)\ni u \mapsto L^*u=\frac{1}{2} \text{tr}\big(\partial_x^2(\sigma\sigma^*u)\big)+\partial_x(b u) .$$ We assume they are uniformly elliptic over  $[0,1]\times \RR^d$, i.e there exists $\kappa>0$ such that for all $(t,x,\xi) \in[0,1]\times \RR^d\times \RR^d$, we have  $\xi^*\sigma\sigma^*(t,x)\xi\ge \kappa|\xi|^2$.	
\end{assumption}

\begin{assumption}[Lyapunov function] \label{asex3c}
There exists a Lyapunov function $W(x): \RR^d\to [0,\infty)$ such that $W\in B_p(\RR^d)$ for some $p\in\NN$, $\underset{|x| \to \infty}{\lim}W(x)=\infty$ and $W\in C^{2}(\RR^d)$.  Moreover, there exists $M_W>0$ such that $|\partial_xW(x)|\le M_W(1+ W(x)), \, \forall x\in \RR^d$ and there exists $M_1>0$ such that $\forall t\in(1,\infty), \,x,y \in \RR^d,$
\[
\big(L(t,x)-L(t,y)\big)W(x-y)\le -M_1W(x-y).
\]
\end{assumption}
Notice that a natural choice for  such a Lyapunov function is $W(x)=|x|^p$ for some $p\ge1$.

Finally, for a fixed value $m\in \NN$, we define our family of test functions
\begin{equation}\label{SL}
\mathcal{S}_m:=\big\{\phi:\RR^d \to \RR \,|\, \phi \in B(\RR^d)\,\, \hbox{and} \,\,  |\phi(x)-\phi(y)|^m\le W(x-y),\,\, \forall x,y \in\RR^d\big\}.
\end{equation}

Next we present a few, even more explicit, examples where Assumptions \ref{asex20}, \ref{asex2} and \ref{asex3c} are satisfied simultaneously.
\begin{exmp}
\item The classical example when $W(x)=|x|^2, \forall x\in\RR^d$ and the drift is allowed to have linear growth is: $0< \sigma \in \RR^d$ and  $[0,\infty)\times \RR^d \ni (t,x) \mapsto b(t,x)=-M_1x$ for some $M_1>0$. Then naturally, Assumptions \ref{asex2a}(b) and \eqref{asex3c} are satisfied on the whole $[0,\infty)\times \RR^d$.
\end{exmp}
\begin{exmp}	
\item In $d=1$, again constant diffusion $\sigma \in [0,\infty)$, $[0,\infty)\times \RR\ni (t,x) \mapsto b(t,x)=\sin(x)-(M_1+1)x$ for some $M_1>0$. This example satisfies Assumption \ref{asex2a}(b) and moreover, by the Mean Value Theorem, Assumption \ref{asex3c} holds with $W(x)=|x|^2, \forall x\in\RR^d$. More explicitly,  for any $t\ge0 $ and $x,y\in \RR$, there exists $\xi \in\RR$ such that
\[
(L(t,x)-L(t,y))W(x-y)=2(b(t,x)-b(t,y))(x-y)=2\partial_xb(t,\xi)(x-y)^2\le -2M_1W(x-y).
\]
\end{exmp}
\begin{exmp}
\item In $d=1, \sigma \in [0,\infty)$ and $[0,\infty)\times \RR\ni (t,x) \mapsto b(t,x)=-xe^{-tx+10(t-0.9)x^2-100(1-t)x^2}$. These coefficients are obviously continuous. Moreover, notice that on one hand, for $t\in[0,1)$ we have that $b(t,x)$ is bounded and therefore this example satisfies Assumption \ref{asex2a}(a). And on the other hand, 
$$\underset{x\in \RR, t\ge 1}{\sup}\,{\partial_xb(t,x)}=\underset{x\in \RR}{\sup}\,{-(1-xt+20x^2(t-0.9)-200x^2(1-t))e^{-tx+10(t-0.9)x^2-100(1-t)x^2}}\le -0.7.$$
And since the diffusion is a constant, Assumption \ref{asex3c} is satisfied for $t\ge 1$ and $W(x)=x^2, \forall x\in\RR$ and $M_1=0.7$.
The idea one must take away from this example is that once the time dependence was the limitation to applying results in the literature, but now the dependence on the time variable is what drives the shift in between the bounded and the  ``decaying as $-x$'' behaviour.
\end{exmp}

\subsection{Finite time estimates on derivatives of transition densities}\label{PDE results}

The forward Kolmogorov equation associated to the process solving the SDE \eqref{original} with  initial time $\tau\ge0$ and initial data $X_{\tau}^{\tau, z}=z\in \RR^d$  is the following PDE 
\begin{equation}\label{fkp}
\partial_tp_t^{\tau}(z,x)-L^*(t,x)p_t^{\tau}(z,x)=0,\quad (t,x)\in [\tau,\infty)\times \RR^d.
\end{equation}
There are many relevant results on well posedness and solution regularity. First, existence of solution follows from the existence of solution to \eqref{original} under Assumptions \ref{asex2a} and \ref{asex3c}. Uniqueness can be guaranteed under one of the conditions in  Assumption \ref{asex2a} (see \cite[Theorem 7.4]{evans1998partial} or \cite{friedman2008partial}). Next we are going to enumerate a few results regarding the stability, regularity and explicit estimates of such solutions collected from the literature on parabolic PDEs and adapt them to the shape of the assumptions in Section \ref{sec main}.

\begin{lemma}\label{fried}
Let Assumptions \ref{asex20} and \ref{asex2a}(a) hold. Assume moreover that  $W$ satisfies that  \\$\underset{c>0}{\sup}\{\int_{\RR^d}|W(x)|e^{-c|x|^2}dx\}<\infty.$ Then, given any $x \mapsto g(x)\le C(1+W(x))^{1/m}$, $C>0$ and  $\phi \in \mathcal{S}_m$ defined in \eqref{SL} for $m\in \NN$,  Assumption \ref{as2} is satisfied with 
$
h(x)=Ce^{c|x|^2}, \,\,\hbox{for all}\quad  x\in \RR^d.
$	 
Moreover, if $W\in B_p(\RR^d)$  then Assumption \ref{as4} holds for some $\delta>0$ and there exist $C,c>0$ such that Assumption \ref{as2} is satisfied with 
\[h(x)=
\begin{cases}
Ce^{c|x|^2}, \quad if \quad p<m\\
C(1+|x|^{p/m}), \quad if \quad p\ge m;\qquad \hbox{for all}\quad  x\in \RR^d.
\end{cases}
\]
\end{lemma}
\begin{proof}
From Assumption \ref{asex20} and Assumption \ref{asex2a}(a), all the conditions in \cite[Theorem 9.4.2, Remark below display (9.4.18)]{friedman2008partial} are satisfied and hence for $0\le n\le 2$,  there exist $C,c>0$ such that:
\begin{equation}\label{f1}
|\partial_x^np_1^0(x,x'')|\le Ce^{-c|x-x''|^{2}},\quad  \forall x,x''\in\RR^d.
\end{equation}
Moreover,  for all $t>1$ and $x,x''\in\RR^d$, 
\begin{equation}\label{f2}
|p_t^1(x,x'')|\le \frac{C}{(t-1)^{d/2}}e^{-c\big(\frac{|x-x''|^2}{(t-1)}\big)}.
\end{equation}
Then  we have directly Assumption \ref{as2} with $\RR^d \ni x\mapsto h(x)=Ce^{c|x|^2}$ (with some other $C,c>0$) since by Young's inequality we have $|x||x''|\le |x|^2+\frac{1}{4}|x''|^2$ and therefore
\begin{align*}
\int_{\RR^d}g(x'')|\partial_x^np_1^0(x,x'')|dx''&\le C\int_{\RR^d}(1+W(x''))^{1/m}e^{-c|x-x''|^{2}}dx''\\
&\le C\int_{\RR^d}(1+W(x''))^{1/m}e^{-c|x|^2-c|x''|^{2}+2c|x||x''|}dx''\\
&\le C\int_{\RR^d}(1+W(x''))^{1/m}e^{c|x|^2-c/2|x''|^{2}}dx''\le  Ce^{c|x|^2}.
\end{align*}
In particular, if $W\in B_p(\RR^d)$, 
\begin{align*}
\int_{\RR^d}(1+W(x''))^{1/m}e^{-c|x''|^{2}}dx''\le C\int_{\RR^d}(1+|x''|^{p/m})e^{-c|x''|^{2}}dx'''\le  C.
\end{align*}
Notice however that if $p>m$ then by a change of variables and noticing that $|x+x''|^{p/m}\le C(|x|^{p/m}+|x''|^{p/m}) $, we can conclude an improved estimate:
\begin{align*}
\int_{\RR^d}g(x'')|\partial_x^np_1^0(x,x'')|dx''&\le \int_{\RR^d}C(1+W(x''))^{1/m}e^{-c|x-x''|^{2}}dx''\le C\int_{\RR^d}(1+|x'|^{p/m}+|x|^{p/m})e^{-c|x'|^{2}}dx'\\
&\le C(1+|x|^{p/m}).
\end{align*} 

We focus now on  Assumption \ref{as4}. Let $\delta>0$ be arbitrary. By Chapman--Kolmogorov identity, Leibniz' formula and  Jensen's inequality:
\begin{align*}
\int_{\RR^d}|\phi(x'')\partial_x^np_t^0(x,x'')|^{1+\delta}dx''
&\le \int_{\RR^{d}}|\phi(x'')|^{1+\delta}\Big|\partial_x^n\int_{\RR^{d}}p_t^1(x',x'')p_1^0(x,x')dx'\Big| ^{1+\delta}dx''\\
&\le \int_{\RR^{d}}|\phi(x'')|^{1+\delta}\Big|\int_{\RR^{d}}p_t^1(x',x'')\partial_x^np_1^0(x,x')dx'\Big| ^{1+\delta}dx''
\\
&\le \int_{\RR^{d}}|\phi(x'')|^{1+\delta}\int_{\RR^{d}}(p_t^1(x',x''))^{1+\delta}|\partial_x^np_1^0(x,x')|^{1+\delta}dx' dx''.
\end{align*}
Moreover, there exists a constant $C_t>0$ (again changing value from line to line but also dependent on $t$) such that  $\underset{x,x''}{\sup}\,p_t^1(x,x'')\le C_t$ as a direct consequence of the results in  \cite{friedman2008partial}. Similarly, there exists $C>0$ such that for $n=1,2$ we have $\underset{x,x'}{\sup}\,\partial_x^np_1^0(x,x')\le C$. Therefore,  if $\phi\in \mathcal{S}_m$ and $W\in B_p(\RR^d)$, we conclude by \eqref{f1} and \eqref{f2}, for all $x\in \RR^d, \,t\ge 1$:
\begin{align*}
\int_{\RR^d}|\phi(x'')\partial_x^np_t^0(x,x'')|^{1+\delta}dx''
&\le (CC_t)^{\delta}\int_{\RR^{d}}\int_{\RR^{d}}|\phi(x'')|^{1+\delta}p_t^1(x',x'')|\partial_x^np_1^0(x,x')|dx'dx''\\
&\le C_t\int_{\RR^{d}}\bigg(\int_{\RR^{d}}(1+|W(x'')|^{(1+\delta)/m})e^{-c\big(\frac{|x'-x''|}{(t-1)}\big)^{2}}dx''\bigg)e^{-c|x-x'|^{2}}dx'\\
&\le C_t\int_{\RR^{d}}\big((1+|W(x')|^{(1+\delta)/m})\big)e^{-c|x-x'|^{2}}dx' <\infty.
\end{align*}

\end{proof}

So far we have shown how to apply Assumption \ref{asex2a}(a). Let us now consider what happens if we wish to use  Assumption \ref{asex2a}(b) instead. The typical example for this scenario, where the coefficients also satisfy Assumption \ref{asex3c} with $W(x)=|x|^2$, is the following: constant $\sigma$ and $b(s,x)=-xe^{-s}$ for all $(s,x)\in(0,\infty]\times\RR^d$. 

An alternative is  the result presented by Eidelman in \cite[Theorem VI.5]{eidelman2012parabolic}. 
\begin{lemma}\label{eild}
Suppose Assumptions \ref{asex20} and  \ref{asex2a}(b) hold. Assume moreover that  $W$ satisfies that  \\$\underset{c>0}{\sup}\,\big\{\int_{\RR^d}|W(x)|e^{-c|x|^2}dx\big\}<\infty.$ Then, given any $x \mapsto g(x)\le C(1+W(x))^{1/m}$, $C>0$ and  $\phi \in \mathcal{S}_m$ defined in \eqref{SL} for some $m\in\NN$,  Assumptions \ref{as4} and \ref{as2} are satisfied for some $C,c, \delta>0$ with $h(x)= Ce^{c|x|^2}, \,\forall x\in \RR^d$. 
\end{lemma}

\begin{proof}

Let $b$ and $ \sigma$ be the coefficients of the SDE (\ref{original}) restricted to the time interval $[0,1]$. Under Assumptions \ref{asex20} and \ref{asex2b}(b)  by \cite[Theorem VI.5]{eidelman2012parabolic}, the fundamental solution $\Gamma$ to the PDE $(L-\partial_t)\Gamma=0$ satisfies the following bounds for some $C,c,\nu>0$ and for $n=0,1,2$:
\begin{equation}\label{f3}
|\partial_x^n\Gamma(0,x'';1,x)|\le C\exp(-c|x-x''|^2+\nu|x|^2-\nu|x''|^2).
\end{equation}
Closer to what we are looking for is the fundamental solution to the adjoint equation \eqref{fkp}, which is denoted by $\Gamma^*$ and  is one time reversal away from $p_{\tau+s}^{\tau}(\xi,x)$. Namely, 
$p_{\tau+s}^{\tau}(\xi,x'')=\Gamma^*(\tau-s,x'';\tau,\xi)=\Gamma(\tau,\xi;\tau-s,x''),$
where the last equality holds by the result proved in \cite[Theorem VI.2]{eidelman2012parabolic} under the additional assumption that the coefficients are H\"older-continuous in space uniformly in time. 

We can therefore conclude that Assumption \ref{as2} holds for  functions $g$ as in the statement of the lemma since  by Young's inequality:
\begin{align*}
\int_{\RR^d}g(x'')|\partial_x^np_1^0(x,x'')|dx''
&\le C\int_{\RR^d}(1+W(x''))^{1/m}e^{-c|x-x''|^2+\nu|x|^2-\nu|x''|^2}dx''\\
&\le Ce^{\nu|x|^2}\Big(\int_{\RR^d}(1+W(x''))^{1/m}e^{-\nu|x''|^2}e^{c|x|^2-c/2|x''|^2}dx''\Big) \\
&\le Ce^{(\nu+c)|x|^2}.
\end{align*}
And in fact, if $W\in B_p(\RR^d)$ then
\begin{align*}
\int_{\RR^d}(1+W(x''))^{1/m}e^{-c|x''|^2}dx''
&\le C\int_{\RR^d}(1+|x''|^{p/m})e^{-c|x''|^2}dx''
\le C.
\end{align*}

We turn now to Assumption \ref{as4} and apply first Chapman--Kolmogorov's identity, Leibniz' formula and Jensen's inequality.  And again,  by the bounds of the derivatives at $t=1$ and the density at any $t\ge 1$ (see \cite{eidelman2012parabolic}), there exists a constant $C_t>0$ dependent on $t$ such that  $\underset{x,x''}{\sup}\,p_t^1(x,x'')\le C_t$. Finally, given the definition of $\mathcal S_m$, there exist $C>0$ such that $\phi\le C(1+W)^{1/m}$. All together, we have that for all $t\ge 1, \, x\in \RR^d$ and by choosing $\delta>0$, there exist $C,C_t>0$ (where again we allow these constants to change from line to line in order to avoid cumbersome notation) such that
\begin{align*}
\int_{\RR^d}&|\phi(x'')\partial_x^np_t^0(x,x'')|^{1+\delta}dx''\\
&\le \int_{\RR^{d}}\int_{\RR^{d}}|\phi(x'')p_t^1(x',x'')|^{1+\delta}|\partial_x^np_1^0(x,x')|^{1+\delta}dx'dx''\\
&\le C^{\frac{(1+\delta)}{m}}C_t^{\delta}\int_{\RR^{d}}\Big(\int_{\RR^{d}}\big(1+W(x'')\big)^{\frac{(1+\delta)}{m}}e^{(-c|x'-x''|^2+\nu|x'|^2-\nu|x''|^2)}dx''\Big)e^{(1+\delta)(-c|x-x'|^2+\nu|x|^2-\nu|x'|^2)}dx'\\
&\le C_t\int_{\RR^{d}}e^{\nu|x'|^2}\Big(\int_{\RR^{d}}\big(1+W(x'')\big)^{\frac{(1+\delta)}{m}}e^{-\nu|x''|^2}e^{-c|x'-x''|^2}dx''\Big) e^{(1+\delta)(-c|x-x'|^2+\nu|x|^2-\nu|x'|^2)}dx'\\
&\le C_t\int_{\RR^{d}}e^{(\nu+c)|x'|^2}e^{-(1+\delta) c|x-x'|^2+(1+\delta)\nu|x|^2-\delta\nu|x'|^2}dx'\\
&\le C_te^{\nu(1+\delta)|x|^2}\Big(\int_{\RR^{d}}e^{(\nu+c-\delta \nu)|x'|^2}e^{-(1+\delta) c|x|^2-(1+\delta) c|x'|^2+2(1+\delta) c|x||x'|}dx'\Big)\\
&\le C_te^{\nu(1+\delta)|x|^2}\Big(\int_{\RR^{d}}e^{(\nu+c-\delta \nu)|x'|^2}e^{(1+\delta) c|x|^2-(1+\delta) c/2|x'|^2}dx'\Big)\\
&\le C_te^{(\nu+c)(1+\delta)|x|^2}<\infty.
\end{align*}
\end{proof}

Note that we could have formulated yet another assumption as an alternative to Assumption \ref{asex2a}(a) and Assumption \ref{asex2a}(b) by employing the result presented by Deck and Kruse \cite[Corrolary 4.2]{deck2002parabolic}. Suppose that Assumption \ref{asex20} holds in addition of $b,\sigma$ belonging to $ C([0,1]\times\RR^d)$, being locally H\"older continuous in both components and globally  H\"older growing in space (i.e in $B_p(\RR^d)$ with $p\in (0,1)$), uniformly in time for the interval $[0,1]$. Deck and Kruse's result allows us to conclude the following type of bounds for the transition density and its derivatives: for some $h, \lambda^*>0, c(h)>0$, $0\le n\le 2$,
$\displaystyle{
|\partial_x^np_s^t(x,y)|\le c(h)e^{hy^2}}\frac{\exp({-\frac{\lambda^*|x-y|^2}{2(t-s)})}}{(t-s)^{1/2}}.$
We omit the detailed computations since two illustrative examples were already presented and having  a detailed catalogue is beyond the scope of this paper.

\subsection{Convergence to a measure under Lyapunov condition}

Consider $\mathcal{S}_m$, the family of test functions defined in \eqref{SL}. If the Lyapunov function is of polynomial order, since the increment of the elements in this family is controlled by it, in Theorem \ref{thex1} we see that having convergence to a static measure in Wasserstein distance (see Lemma \ref{WTV}) is enough to guarantee Assumption \ref{as3}.

\begin{lemma}\label{WTV}
Let Assumption \ref{asex3c} hold, $\sigma$ be bounded and $\xi$ a random variable such that  $\LL(\xi)\in \mathcal{P}^{W^2}(\RR^d)$. Assume moreover that there exists $c_W>0$ such that $c_W|x|^p\le W(x), \, \forall x\in \RR^d$. Then there exists a measure $q\in \mathcal{P}^W(\RR^d)$ and a constant $C>0$ such that the density $x''\mapsto p_{1+s}^1(x,x'')$ of the random variable $X_{1+s}^{1,x}$ defined by \eqref{original} satisfies the following:
\begin{equation}\label{WTVeq}
\underset{\phi\in \mathcal{S}_m}{\sup}\bigg|\int_{\RR^{d}}\phi(x')(p_{1+s}^1(x,dx')-q(dx'))\bigg|\le  \Big(e^{-M_1 s}\int_{\RR^{d}} W(x-x'')q(dx'')\Big)^{1/m},\quad for\, all \,\, (s,x)\in(0,\infty)\times \RR^d.
\end{equation}
\end{lemma}

\begin{proof}
Let $\lambda$ be an arbitrary positive constant and $\xi, \xi'$ two independent, distinct starting random variables. If It\^o's formula is applied to the stochastic process $\big(e^{\lambda s}W(X_s^{1,\xi}-X_s^{1,\xi'})\big)_{s\ge1}$, we obtain:
\begin{align*}
d\big(e^{\lambda s}W(X_s^{1,\xi}-X_s^{1,\xi'})\big)&=e^{\lambda s}\bigg(\lambda(W(X_s^{1,\xi}-X_s^{1,\xi'}))+(\partial_xW(X_s^{1,\xi}-X_s^{1,\xi'}))\big(b(s,X_s^{1,\xi})-b(s,X_s^{1,\xi'})\big)\\
&\quad +\frac{1}{2}\hbox{tr}\big((\sigma^*(s,X_s^{1,\xi})-\sigma^*(s,X_s^{1,\xi'})\partial_x^2W(X_s^{1,\xi}-X_s^{1,\xi'})(\sigma(s,X_s^{1,\xi})-\sigma(s,X_s^{1,\xi'}))\big)\bigg)ds\\
&\qquad +e^{\lambda s}\bigg(\partial_xW(X_s^{1,\xi}-X_s^{1,\xi'})\big(\sigma(s,X_s^{1,\xi})-\sigma(s,X_s^{1,\xi'})\big)\bigg)dB_s.
\end{align*}
At this point we do not know whether the process defined for any $t> 1$ as it follows
\begin{equation}\label{integral_process}
\mathbb{Y}_t^{1,\xi}:=\int_1^{t}e^{\lambda s}\bigg(\partial_xW(X_s^{1,\xi}-X_s^{1,\xi'})\big(\sigma(s,X_s^{1,\xi})-\sigma(s,X_s^{1,\xi'})\big)\bigg)dB_s
\end{equation}
is a true martingale. However, we can proceed assuming it for now and we will conclude the proof with the necessary stopping time argument. 
We continue by taking expectations of the above and using Assumption \ref{asex3c} to arrive to 
\begin{align}\label{moment111}
\nonumber d\big(e^{\lambda (s+1)}\EE[W(X_{1+s}^{1,\xi}-X_{1+s}^{1,\xi'})]\big)&=e^{\lambda ({s+1})}\big(\lambda\EE[(W(X_{1+s}^{1,\xi}-X_{1+s}^{1,\xi'}))]\\\nonumber 
&\qquad +\EE[(L({s+1},X_{1+s}^{1,\xi})-L({s+1},X_{1+s}^{1,\xi'}))W(X_{1+s}^{1,\xi}-X_{1+s}^{1,\xi'})]\big)ds\\
&\le e^{\lambda ({s+1})}(\lambda-M_1)\EE[W(X_{1+s}^{1,\xi}-X_{1+s}^{1,\xi'})] \,ds.
\end{align}
Therefore, for the particular choice of $\lambda=M_1>0$, the above means $d\big(e^{\lambda ({s+1})}\EE[W(X_{1+s}^{1,\xi}-X_{1+s}^{1,\xi'})]\big)\le 0 \,ds$, which by integrating in between $1$ and $s+1>1$ implies that 
\begin{align}\label{W-moment_increment}
\EE\big[W(X_{1+s}^{1,\xi}-X_{1+s}^{1,\xi'})\big]\le e^{-M_1 s}\EE[W(\xi-\xi')].
\end{align}
Next recall that we assumed the existence of $c_W>0$ such that $c_W|x|^p\le W(x), \, \forall x\in \RR^d$. Moreover, by the properties of $W$ and the Wasserstein distance  we have the following bound from \eqref{W-moment_increment}:
\begin{align}
\nonumber c_W\mathcal{W}_p^p(\LL(X_{1+s}^{1,\xi}),\LL(X_{1+s}^{1,\xi'}))&\le  c_W\EE\big[|X_{1+s}^{1,\xi}-X_{1+s}^{1,\xi'}|^p\big]\le  \EE\big[W(X_{1+s}^{1,\xi}-X_{1+s}^{1,\xi'})\big]\le e^{-M_1 s}\EE[W(\xi-\xi')] \\\nonumber
&\le e^{-M_1 s}(1+\EE[|\xi-\xi'|^p])\le Ce^{-M_1 s}\EE[|\xi-\xi'|^p].
\end{align}
Notice now that by the definition of infimum, for every $\epsilon>0$ there exists a coupling $\Pi_{\epsilon}$ in between $\LL(\xi)$ and $\LL(\xi')$ which satisfies:
\begin{align}
\nonumber\int_{\RR^d} \int_{\RR^d}|x-y|^p \Pi_{\epsilon}(dx,dy)&\le  \mathcal{W}_p^p(\LL(\xi),\LL(\xi'))+\epsilon.
\end{align}
And by a simple construction we know that if enlarged adequately the space (where we still denote the whole expectation by $\EE$), there exist $\tilde \xi $ and $\tilde \xi'$ independent copies of $\xi$ and $\xi'$ respectively such that $\EE[|\tilde\xi-\tilde\xi'|^p]	=\int_{\RR^d} \int_{\RR^d}|x-y|^p \Pi_{\epsilon}(dx,dy)$ . All together,
\begin{align}
\nonumber \mathcal{W}_p^p(\LL(X_{1+s}^{1,\xi}),\LL(X_{1+s}^{1,\xi'}))&\le  \frac{Ce^{-M_1 s}}{c_W}\EE[|\xi-\xi'|^p]=\frac{Ce^{-M_1 s}}{c_W}\EE[|\tilde\xi-\tilde\xi'|^p]= \frac{Ce^{-M_1 s}}{c_W}\int_{\RR^d} \int_{\RR^d}|x-y|^p \Pi_{\epsilon}(dx,dy)\\\nonumber
&\le \frac{Ce^{-M_1 s}}{c_W}\mathcal{W}_p^p(\LL(\xi),\LL(\xi'))+\epsilon.
\end{align}
Since this holds for arbitrary $\epsilon>0$, we conclude
\begin{align}\label{contraction Lyapunov}
\mathcal{W}_p(\LL(X_{1+s}^{1,\xi}),\LL(X_{1+s}^{1,\xi'}))&\le \frac{C^{1/p}e^{-\frac{M_1}{p}s}}{c_W^{1/p}}\mathcal{W}_p(\LL(\xi),\LL(\xi')).
\end{align}

In particular, this bound implies  convergence to invariant measure for the process $X^{1,\xi}$ given that $\LL(\xi)\in \mathcal{P}^W(\RR^d)$.
Indeed, first notice that the space $\mathcal{P}^W(\RR^d)$ equipped with the $p$--Wasserstein distance is a closed subspace $\big($subspace because of the polynomial growth and properties of the increments of $W$; and closed by definition of $\mathcal{P}^W$ and polynomial growth of $W$$\big)$ of the complete metric space of the probability measures with finite $p$--moments and the $p$--Wasserstein distance (see for e.g. \cite{villani2003topics}). Consequently, $(\mathcal{P}^W(\RR^d),\mathcal{W}_p)$ is a complete metric space itself.  Moreover, for $\hat{s}\ge0$ large enough such that $\frac{Ce^{-M_1\hat s}}{c_W}<1$, inequality \eqref{contraction Lyapunov} implies contraction for the map $\mathcal T_{\hat{s}+1}(\mathcal{L}(\xi)):=\mathcal{L}(X_{\hat{s}+1}^{1,\xi})$ defined in this space. Therefore the Banach Fixed Point Theorem asserts the existence of a unique fixed point $\hat{q}$ for $\mathcal T_{\hat{s}+1}$, which at this stage could be dependent on $\hat{s}$.
However, one can construct another measure $q:=\int_0^{\hat{s}}\mathcal T_s(\hat q)ds$ which, by a similar argument to that in \cite{hu2019mean}, we next prove to be an invariant measure for $\mathcal T_{s}$ with arbitrary $s\in [0,\infty)$. 

Consider first $r\le \hat{s}$ and notice that, from the properties of the semigroup, the new measure $q:=\int_0^{\hat{s}}\mathcal T_s(\hat q)ds$, still belonging to $\mathcal{P}_p(\RR^d)$  (and therefore to $\mathcal{P}^W(\RR^d)$), satisfies:
\[
\mathcal T_rq=\mathcal T_r\int_0^{\hat{s}}\mathcal T_s(\hat q)ds=\int_0^{\hat{s}}\mathcal T_r\mathcal T_s(\hat q)ds=\int_0^{\hat{s}}\mathcal T_{s+r}(\hat q)ds=\int_r^{\hat{s}+r}\mathcal T_s(\hat q)ds=\int_r^{\hat{s}}\mathcal T_s(\hat q)ds+\int_{0}^{\hat{s}+r}\mathcal T_s(\hat q)ds.
\]
Recalling now that $\hat q$ is a fixed point of $T_{\hat{s}}$, we conclude that indeed $q$ is a fixed point for $T_r$ for all $0\le r\le \hat{s}$:
\[
\mathcal T_rq=\int_r^{\hat{s}}\mathcal T_s(\tilde q)ds+\int_{\hat{s}}^{\hat{s}+r}\mathcal T_s(\tilde q)ds=\int_r^{\hat{s}}\mathcal T_s(\tilde q)ds+\int_{0}^{r}\mathcal T_s(\tilde q)ds=\int_0^{\hat{s}}\mathcal T_s(\tilde q)ds=q.
\]

Next consider the other possibility of $\hat{s}< r$. If this is the case, there exists a natural number $k$ such that $k\hat{s}<r\le (k+1)\hat{s}$. Hence, notice that a simple iteration allows us to step on the previous case:
\[\mathcal T_rq=\mathcal T_{k\hat{s}}\mathcal T_{r-k\hat{s}}q=(\mathcal T_{\hat{s}})^k\mathcal T_{r-k\hat{s}}q=q.\]
We prove therefore that there must exist a fixed point for $(\mathcal T_s)_{s\ge 0}$, which in other words means an invariant measure for the process $(X_s)_{s\ge 0}$.

The uniqueness of such an invariant measure follows from \eqref{contraction Lyapunov} . Indeed, when $\xi$ is taken to be distributed as another invariant measure  $\mathcal{L}(\xi)\not= q$, from \eqref{contraction Lyapunov} 
\[
\mathcal{W}_p(\mathcal{L}(\xi),q)=\mathcal{W}_p(\mathcal{L}(X_{1+s}^{\xi}),\mathcal{L}(X_{1+s}^{\xi'}))\le Ce^{-\frac{M_1}{p}s}\mathcal{W}_p(\mathcal{L}(\xi),q),
\]
and since this must hold for all $s\ge 0$, we arrive to the contradiction that $\xi'\sim \xi,$ concluding that way uniqueness of the invariant measure $q$. 

Moreover, if we choose $\xi'$ distributed as the invariant measure denoted by $q$ and we take test functions $\phi$ belonging to $\mathcal{S}_m$ defined in \eqref{SL}, display \eqref{W-moment_increment} implies after using H\"older's inequality that: 
\begin{align*}
\underset{\phi \in \mathcal{S}_m}{\sup}&\bigg|\int_{\RR^{d}}\phi(x') (p_{1+s}^1(x,dx')-q(dx'))\bigg|
= \underset{\phi \in \mathcal{S}_m}{\sup}\bigg|\int_{\RR^{d}} \phi(x')p_{1+s}^1(x,dx')-\int_{\RR^{d}} \phi(x'')q(dx'')\bigg|\\
&=\underset{\phi \in \mathcal{S}_m}{\sup}\big|\EE\big[\phi(X_{1+s}^{1,x})]-\EE[\phi(X_{1+s}^{1,\xi'})\big]\big|=\underset{\phi \in \mathcal{S}_m}{\sup}\big|\EE\big[\phi(X_{1+s}^{1,x})-\phi(X_{1+s}^{1,\xi'})\big]\big|\\
&\le\underset{\phi \in \mathcal{S}_m}{\sup}\EE\big[|\phi(X_{1+s}^{1,x})-\phi(X_{1+s}^{1,\xi'})|\big]\le \underset{\phi \in \mathcal{S}_m}{\sup}\big(	\EE\big[\big|\phi(X_{1+s}^{1,x})-\phi(X_{1+s}^{1,\xi'})\big|^m\big]\big)^{1/m}\\
&\le\big(\EE\big[W(X_{1+s}^{1,x}-X_{1+s}^{1,\xi'})\big]\big)^{1/m}\le e^{-M_1/m s}(\EE[W(x-\xi')])^{1/m}\\
&=e^{-\frac{M_1}{m} s}\bigg(\int_{\RR^{d}} W(x-x'')q(dx'')\bigg)^{1/m}.
\end{align*}
Hence, we proved that \eqref{WTVeq} holds.

We conclude the proof by illustrating the stopping time argument. Let us define the following increasing sequence of stopping times: $T_n:=\inf\{t\ge1\, :\, |X_t^{\xi}|>n \}\to \infty,\, n\to \infty.$  Notice that the stopped version of the integral process \eqref{integral_process} is a uniformly integrable martingale. For the process in \eqref{integral_process} to be a true martingale it is enough for the integrand to be square integrable. 
Since on one hand $\sigma$ is  bounded and on the other hand there exists $M_W>0$ such that $|\partial_xW(x)|\le M_W(1+ W(x)), \, \forall x\in \RR^d$ 
(see Assumption \ref{asex3c})
we get for all $t>1$:
\begin{align*}
\EE\bigg[\int^{t\wedge T_n}_1&\bigg(e^{\lambda s}\partial_xW(X_s^{1,\xi}-X_s^{1,\xi'})\big(\sigma(s,X_s^{1,\xi})-\sigma(s,X_s^{1,\xi'})\big)\bigg)^2ds\bigg]\\
&\le2M^2\EE\bigg[\int^{t}_1e^{2\lambda (s\wedge T_n)}|\partial_xW(X_{s\wedge T_n}^{1,\xi}-X_{s\wedge T_n}^{1,\xi'})|^2ds\bigg]\\
&\le 2M^2M_W^2\EE\bigg[\int^{t}_1e^{2\lambda (s\wedge T_n)}(1+W(X_{s\wedge T_n}^{1,\xi}-X_{s\wedge T_n}^{1,\xi'})^2)ds\bigg].
\end{align*}
Moreover, the stopped process defined from \eqref{integral_process} is a martingale and when applied to it, inequality \eqref{W-moment_increment} implies 
\begin{align*}
2M^2M_W^2\EE\bigg[\int^{t}_1e^{2\lambda (s\wedge T_n)}(1+W(X_{s\wedge T_n}^{1,\xi}-X_{s\wedge T_n}^{1,\xi'})^2)ds\bigg]\le C\EE\bigg[\int^{t}_1e^{2\lambda s}(1+W(\xi-\xi'))^2ds\bigg].
\end{align*}
Given this uniform in $n$  boundedness of the right hand side, when letting $n\to \infty$ in the inequality above  we conclude by Fatou's Lemma (allowing us to take the limit inside the expectation on the left hand side) that  
\begin{align*}
\EE\bigg[&\int^{t}_1\bigg(e^{\lambda s}\partial_xW(X_s^{1,\xi}-X_s^{1,\xi'})\big(\sigma(s,X_s^{1,\xi})-\sigma(s,X_s^{1,\xi'})\big)\bigg)^2ds\bigg]\\
&\le\underset{n\to \infty}{\liminf}\,\EE\bigg[\int^{t\wedge T_n}_1\bigg(e^{\lambda s}\partial_xW(X_s^{1,\xi}-X_s^{1,\xi'})\big(\sigma(s,X_s^{1,\xi})-\sigma(s,X_s^{1,\xi'})\big)\bigg)^2ds\bigg] \\
&\le  C\EE[W^2(\xi-\xi')]\int_1^te^{2\lambda s}ds<\infty.
\end{align*}
Hence, the stochastic integral in \eqref{integral_process} is a martingale. 
\end{proof}

\subsection{Result for autonomous SDEs}
We are now ready to combine the results from the previous sections and apply Theorem \ref{main} to the example falling under the set of Assumptions \ref{asex2a}--\ref{asex20}--\ref{asex3c}.

\begin{theorem}\label{thex1}

Suppose that Assumptions \ref{asex20}, \ref{asex3c} and \ref{asex2a}(a) or \ref{asex2b}(b) hold.  Assume moreover that there exists $c_W>0$ such that $c_W|x|^p\le W(x), \, \forall x\in \RR^d$ for the constant $p$ in Assumption \ref{asex3c}. Then for $n=1,2$  and for any $\phi \in \mathcal{S}_m$ defined in \eqref{SL}:
\[
|\partial_x^nV(s,x)|\le e^{(-M_1/m) s}h(x),\quad for\, all \, (s,x)\in(1,\infty)\times \RR^d,
\]
where for some $C,c>0$, we have 
\[
h(x)=
\begin{cases}
Ce^{c|x|^2}, \qquad \qquad if \quad p<m,\\
C(1+|x|^{p/m}) , \quad if \quad p\ge m;
\end{cases}
\] under Assumption \ref{asex2a}(a) or $h(x)= Ce^{c|x|^2}$ under Assumption \ref{asex2b}(b).
\end{theorem}
\begin{proof}
By Lemma \ref{WTV}, we have seen that under  Assumption \ref{asex3c},  we have
\begin{align*}
\underset{\phi\in \mathcal{S}}{\sup}\Big|\int_{\RR^{d}}\phi(x')(p_{1+s}^1(x,dx')-q(dx'))\Big|
&\le  e^{-(M_1/m)s}\Big(\int_{\RR^{d}} W(x-x')q(dx')\Big)^{1/m}.
\end{align*}
And since
\[
\int_{\RR^d}W(x-x'')q(dx'')\le \int_{\RR^d}C(1+|x-x''|^{p})q(dx'') \le C(1+|x|^{p}),
\] we know that Assumption \ref{as3} is satisfied with $g(x)= C(1+ |x|^{p/m})$ and $G(s)=e^{-(M_1/m)s}$. 

Moreover, under Assumptions \ref{asex20} and  \ref{asex2a}(a), recall that Lemma \ref{fried} asserts that Assumption \ref{as2} is satisfied with $h(x)= Ce^{c|x|^2}$ since by Young's inequality:
\begin{align*}
\int_{\RR^d}(1+ |x''|^{p/m})Ce^{-c|x-x''|^{2}}dx''&\le\int_{\RR^d}(1+ |x''|^{p/m})Ce^{-c|x|^2-c|x''|^{2}+2c|x||x''|}dx''\\&\le\int_{\RR^d}(1+ |x''|^{p/m})Ce^{c|x|^2-c/2|x''|^{2}}dx''\\
&\le  Ce^{c|x|^2}.
\end{align*}
Now, if $p\ge m$ by a mere change of variables, we conclude on a similar fashion to Lemma \ref{fried} that the bound is achieved with $h(x)=C(1+|x|^{p/m})$.

Alternatively, under Assumptions \ref{asex20} and  \ref{asex2b}(b), in Lemma \ref{eild} we prove that  Assumption \ref{as2}(i) is satisfied with $h(x)=  Ce^{(\nu+c)|x|^2}$.

Moreover, Assumption \ref{as4} is proved under both  Assumptions \ref{asex2a}(a) and  \ref{asex2a}(b) again in Lemmas \ref{fried} and \ref{eild} respectively.

Finally, since Assumptions \ref{as3}, \ref{as4} and \ref{as2} are satisfied, Theorem \ref{main} gives us the claim.
\end{proof}
\subsection{Monotonic case}\label{partic_case}

In this section we present a more restrictive, although also more intuitive, set of assumptions for guaranteeing the conclusion of Theorem \ref{thex1}.  Instead of the abstract Lyapunov functions we work with in Theorem \ref{thex1}, we work directly with $x \mapsto |x|^2$.

\begin{assumption}[Monotonicity]\label{asex3d}
For any $(s,x,x_0)\in [0,\infty)\times\RR^d\times\RR^d$ there exist some  $ M_1 >0$  such that 
\[
\langle x-x_0,b(s,x)-b(s,x_0) \rangle +\frac{m-1}{2}|\sigma(s,x)-\sigma(s,x_0)|^2\le -M_1|x-x_0|^2.
\]	
\end{assumption}

Moreover, for some $m\in \NN$ such that $p\ge m\ge 2$, we take test functions within the $m^{th}$--order locally Lipschitz functions:
\begin{equation}\label{Smonotone}
\mathcal{S}_m':=\{\phi:\RR^d \to \RR \,|\, \phi \in B(\RR^d)\, \hbox{and} \,\exists C>0\, s.t\ |\phi(x)-\phi(y)|\le C(1+|x|^{m/2}+|y|^{m/2})|x-y|,\, \forall x,y \in\RR^d\}.
\end{equation}

We say more intuitive because the monotonicity and locally Lipschitz conditions are popular assumptions in the field. We also see next that this example is a subset of that presented up until now in this section. 
This means that,  as a particular case of Theorem \ref{thex1}, the same type of decaying in time estimates hold for this scenario.
\begin{corollary}\label{thex3}
Suppose that Assumptions \ref{asex20}, \ref{asex3d} and Assumption \ref{asex2a}(a) or \ref{asex2b}(b) hold.  Then for $n=1,2$  and for any $\phi \in \mathcal{S}_m'$ defined in \eqref{Smonotone}, the following holds:
\[
|\partial_x^nV(s,x)|\le e^{-M_1 s}h(x),\quad  \forall (s,x) \in(1,\infty)\times\RR^d;
\]
where for some $C>0$, $h(x)=C(1+|x|^m)$ under Assumption \ref{asex2a}(a) or $h(x)=C(1+|x|^m)e^{\nu|x|^2}$ under Assumption \ref{asex2b}(b).
\end{corollary}	
\begin{proof}
First notice that when we have the Lyapunov condition Assumption \ref{asex3c} holding for $W(x)=|x|^{p}$, the monotonicity  condition is satisfied  automatically. Moreover, for $\phi\in \mathcal{S}_m$ with $\frac{p}{m}\ge 2$,
\begin{align*}
|\phi(x)-\phi(y)|&\le |x-y|^{\frac{p}{m}}=|x-y|^2|x-y|^{\frac{p}{m}-2}\le(|x|+|y|)|x-y||x-y|^{\frac{p}{m}-2}\\
&\le...\le (|x|^{\frac{p}{m}-1}+|y|^{\frac{p}{m}-1})|x-y|,
\end{align*}
i.e $\phi\in \mathcal{S}_m$ and $W(x)=|x|^{p}$ in Assumption \ref{asex3c}, implies $\phi\in \mathcal{S}_{2(p/m-1)}'$.
This and Theorem \ref{thex1} give the claim.
\end{proof}

Let us conclude with a few remarks on the relation in between the the family of test functions $\mathcal{S}_m$ and $\mathcal{S}_m'$. 
First, the family $\mathcal{S}_m$ is significantly wider than that of $\mathcal{S}_m'$. For one, if $W$ is a polynomial of order $p$, the ratio in between $p$ and $m$ will determine if the admissible test functions in $\mathcal{S}_m$ are either locally Lipschitz or locally H\"older, while the functions in $\mathcal{S}_m'$ must be locally Lipschitz.
If instead $W$ is polynomially growing only, then the test functions in $\mathcal{S}_m$ can have additionally a bounded component added to this locally Lipschitz or locally H\"older part. 

Moreover, although the locally Lipschitz condition seems more natural than the one presented through the Lyapunov function,  it doesn't give itself easily to generalization. Indeed, the family $\mathcal S_m$ can used with other than polynomial  Lyapunov functions  as long as one can obtain contraction in $W$--Weighted Total Variation distance and $W$ is still integrable against the bounds obtained for the derivatives of the transition density (see Appendix \ref{auxintuition}). 

\newpage
\section{Application to non--autonomous SDE decaying to an autonomous one}
\label{auxintuition}

Consider again the non--autonomous SDE \eqref{original}. Given the time--varying character of the coefficients, these processes do not easily have an invariant measure. Recall however that Assumption \ref{as3} was not asking for one, just for a static measure to which the transition probabilities ``stick''.

Additional intuition behind the result in Theorem \ref{main} could be extracted when noticing that  if there exist limiting functions $b_{\infty}(x):=\lim_{r\to \infty}b(x,r)$ and $\sigma_{\infty}(x):=\lim_{r\to \infty}\sigma(x,r)$, we can define the auxiliary process $Z$ solution to the widely studied autonomous SDE:
\begin{equation}\label{limiting}
dZ_s^{\tau,x}=b_{\infty}(Z_s^{\tau,x})ds+\sigma_{\infty}(Z_s^{\tau,x})dB_s;\,\,\, Z_{\tau}^{\tau,x}=x.
\end{equation}
We will formulate conditions so that $Z$ has an invariant measure $q$ which, given its autonomous character, are quite lax. Moreover, the transition densities of the process $X$ decay to $q$.
This together with the same regularity of $b, \sigma$ from Section \ref{non_autonomous} will let us verify Assumptions \ref{as2} and \ref{as3} and conclude decay of the derivatives in space of $V(t,x)=\EE[\phi(X_t^x)]$. Moreover,  we will show that unlike in Section \ref{non_autonomous}, the test functions $\phi$  are not restricted to polynomial growth but are rather controlled by a Lyapunov function which we specify later. 

In order to do so, we cannot longer use the Wasserstein metric to obtain  convergence to invariant measure of $Z$. The  Weighted Total Variation (WTV) distance (see \eqref{WTV_def} for the definition and  \cite[Theorem 8.9]{villani2008optimal,bellet2006ergodic} for more detail) proves to be more appropriate  in this setting. However, note that for polynomial weight functions of degree $p$, it is equivalent to the $p$-- Wasserstein metric (see \cite[Theorem 6.15]{villani2008optimal}). Following \cite{bogachev2019convergence}, we prove in Proposition \ref{WTV_for good} that under another type of Lyapunov condition (see Assumption \ref{asex3cc}), one obtains Assumption \ref{as3} by means of convergence in WTV.

Let us quantify now the above qualitative statements. A possible set of assumptions, in addition to Assumptions \ref{asex2a} and \ref{asex20} is the following.
\begin{assumption}[Time dependence of the coefficients]\label{asexaintuition}\label{coeff}
There exist $\RR^d\ni x \mapsto b_\infty(x)\in \RR^d$ and $\RR^d\ni x \mapsto \sigma_\infty(x)\in \RR^{d\times d}$  measurable such that 
\begin{equation*}
b_{\infty}(x):=\lim_{r\to \infty}b(x,r)\quad \hbox{and} \quad \sigma_{\infty}(x):=\lim_{r\to \infty}\sigma(x,r),\quad   \forall(t,x)\in(0,\infty)\times\RR^d.
\end{equation*}
\end{assumption}
Notice that as a consequence of this and Assumptions \ref{asex2a} and \ref{asex20}, there exists  $L_{\infty}(x):=\lim_{r\to \infty}L(r,x), \forall x\in \RR^d$ which also inherits regularity and uniform ellipticity from $L$.

\begin{assumption}[Lyapunov function] \label{asex3cc}
There exists a  function $W: \RR^d\to [1,\infty)$ such that $\underset{|x| \to \infty}{\lim}W(x)=\infty$ and  it satisfies for some $M_2,>0$ that 
\begin{equation}\label{L_first}
\underset{t\ge 0}{\sup}\{|b(t,x)|,|\sigma(t,x)|\}\le M_2W(x), \quad \forall x\in \RR^d.
\end{equation}
Additionally, we assume $W\in C^{2}(\RR^d)$ and there exists a constant $M_1>0$ such that $\forall t\ge 0, \,x \in \RR^d,$
\begin{equation}\label{L}
L(t,x)W(x)\le -M_1W(x).
\end{equation}
\end{assumption}
We can see that \eqref{L_first} and \eqref{L} also hold with $b,\sigma,L$ replaced by $b_{\infty},\sigma_{\infty},L_{\infty}$.  Hence, \eqref{limiting} has a unique weak solution by \cite[Theorem 2.4]{gyongy1996existence}.

Moreover, we work with the family of test functions:
\begin{equation}\label{SLL}
\mathcal{S}:=\Big\{\phi:\RR^d \to \RR \,|\, \phi \in B(\RR^d) \, \hbox{and} \, \underset{x\in \RR^d}{\sup}\frac{|\phi(x)|}{W(x)}<\infty \Big\}.
\end{equation}

\begin{remark}
Although this set of assumptions might seem restrictive at first sight, it is more general than what is available in the current  literature. Indeed, in addition to the generalization to the non--autonomous dynamics, when compared to \cite{talay1990second}, we notice one extra advantage of the main results of this section, Theorem \ref{thex1_for_good}: there is no need for smoothness and boundedness of all derivatives for the coefficients. In other words, even when discussing the time decay of the space derivatives of $(t,x)\mapsto \EE[\phi(Z_t^{0,x})]$, our results hold under weaker assumptions over the coefficients $b_\infty, \sigma_\infty$.
\end{remark}

Next we state a technical lemma which gives a priori uniform estimates ensuring well--posedness and is used to make the jump form the auxiliary autonomous process $Z$ to the original non--autonomous process $X$.
\begin{lemma}\label{WTVmoments}
Let Assumption \ref{asex3cc} hold and $\sigma$ be bounded. Then, for all $\tau\ge0$ and  $\xi$ satisfying $\EE[W(\xi)]<\infty$,  the processes $(X_{\tau+s}^{\tau,\xi})_{s\ge0}$, $(Z_{\tau+s}^{\tau,\xi})_{s\ge0}$ solutions to \eqref{original} and  \eqref{limiting} respectively, satisfy the following: 
\[
\EE\big[W(X_{\tau+s}^{\tau,\xi})\big]\le  e^{-M_1 s}\EE[W(\xi)]\quad \hbox{and} \quad \EE\big[W(Z_{\tau+s}^{\tau,\xi})\big]\le  e^{-M_1 s}\EE[W(\xi)], \qquad \forall s>0.
\]
\end{lemma}
\begin{proof}
First notice that, given conditions on $W$ presented in Assumption \ref{asex3cc},  $L(t,x)W(x)\le -M_1W(x)$.
An analogous computation to Lemma \ref{WTV} applying  It\^o's formula  to the process $(e^{\lambda s}W(X_{s+\tau}^{\tau,\xi}))_{s\ge0}$ for arbitrary $\lambda\in \RR$, leads to:
\begin{align*}
d\big(e^{\lambda s}\EE{\big[W(X_{s+\tau}^{\tau,\xi})\big]}\big)\le e^{\lambda s}(\lambda-M_1)\EE{\big[W(X_{s+\tau}^{\tau,\xi})\big]}\,ds.
\end{align*} 
Hence, for $\lambda=M_1$ and for all $s>0$, after integrating the previous inequality, we obtain the claimed bound: 
$$\EE{\big[W(X_{s+\tau}^{\tau,\xi})\big]}\le \EE[W(\xi)]e^{-M_1s}.$$
Moreover, since as a consequence of Assumption \ref{coeff}, inequality \eqref{L} is also satisfied by the limiting generator $L_\infty$, we have the analogous is satisfied also by $Z$:
$$\EE{\big[W(Z_{s+\tau}^{\tau,\xi})\big]}\le \EE[W(\xi)]e^{-M_1s}.$$
\end{proof}

Recall that the Weighted Total Variation norm and distances, represented by $||\cdot||_W$ and $d_W$, are stated in \eqref{WTV_def}.
\begin{proposition}\label{WTV_for good}
Let Assumptions \ref{asex2a}, \ref{asex20}, \ref{coeff} and \ref{asex3cc} hold and consider $\xi$ a random variable such that  $\LL(\xi)\in \mathcal{P}^W(\RR^d)$.  Let us denote for any $x\in \RR^d$ the density of $(Z_{1+s}^{1,x})_{s\ge0}$ (solution to \eqref{limiting}) by $\RR^d \ni x''\mapsto p_{1+s}^1(x,x'')$. Then, $(Z_{1+s}^{1,\xi})_{s\ge0}$  has an invariant measure $q\in \mathcal{P}^W(\RR^d)$ and  the following is satisfied for  some $C,c>0$ :
\begin{align}\label{contraction Lyapunov_in_weighted}
d_W(\LL(Z_{1+s}^{1,\xi}),q)&\le Ce^{-c s}d_W(\LL(\xi),q).
\end{align}
Moreover, for any $\mathcal{S}$ defined by \eqref{SLL}, there exist $C,c>0$ such that
\begin{equation}
\underset{\phi\in \mathcal{S}}{\sup}\bigg|\int_{\RR^{d}}\phi(x')(p_{1+s}^1(x,dx')-q(dx'))\bigg|\le  Ce^{-c s}\int_{\RR^{d}} W(x-x'')q(dx''),\forall \,\, (s,x)\in[1,\infty)\times \RR^d.
\end{equation}
\end{proposition}
\begin{proof}
We follow the argument from \cite{bogachev2019convergence}.

\textbf{Step 1:} For the purposes of this proof let us introduce the semigroup $(\mathcal T_{1+s})_{s\ge0}$ defined by the generator $C^2(\RR^d)\ni u \mapsto L_\infty u=\frac{1}{2} \text{tr}\big(\sigma_\infty\sigma_\infty^*\partial_x^2u\big)+b_\infty\partial_xu$. Let us also assume for now that there exists an invariant measure $q\in\mathcal{P}^W(\RR^d)$, and we will prove this claim later.  We know that $(\mathcal T_{1+s})_{s\ge0}$  indeed exists, is unique and is a Markov semigroup on $L^1(\LL(\xi))$ (see \cite[ Theorem 5.2.2,
Proposition 5.2.5 and Example 5.5.1]{bogachev2015fokker}). By \cite[Theorem 6.4.7]{bogachev2015fokker} we know that there exists a positive continuous function $(t,x,x'')\mapsto \rho_{1+t}(x,x'')$ such that for any $\psi\in L^1(q)$ the following identity holds: $\mathcal T_{1+t}\psi(x)=\int_{\RR^d}\rho_{1+t}(x,x'')\psi(x'')dx'', \, \forall x\in \RR^d$. Moreover, for any fixed $x\in \RR^d$, it satisfies $\partial_t\rho_{1+t}(x,x'')=L^*_\infty(x,x'')\rho_{1+t}(x,x'')$ for all $t\ge0, x''\in \RR^d$.

We know moreover that fixed $x\in \RR^d$, the density $(s,x'')\mapsto p_{1+s}^1(x,x'')$ satisfies the Fokker--Planck equation:
\[
\partial_tp_{1+s}^1(x,x'')=L^*_\infty(x,x'')p_{1+s}^1(x,x'') ;  \quad p_{1}^1(x,x'')=\delta_x(x''), \,\, s\ge0, x''\in \RR^d.
\]
By uniqueness of solution under our regularity assumptions on the coefficients we know that $\rho_{1+s}(x,x'')=p_{1+s}^1(x,x'')$ and therefore for any $\psi\in L^1(q)$, $\mathcal T_{1+t}\psi(x)=\int_{\RR^d}p_{1+t}^1(x,x'')\psi(x'')dx'', \, \forall x\in \RR^d, t\ge 0$.  Moreover, its dual is $\mathcal T_{1+t}^*\sigma(x'')=\int_{\RR^d}p_{1+t}^1(x,x'')\sigma(dx), \, \forall x''\in \RR^d, t\ge 0$ and $\sigma \in \mathcal{P}(\RR^d)$.

Notice next that by Lemma \ref{WTVmoments} we know that for all $(s,x)\in (0,\infty)\times\RR^d $:
\begin{equation}\label{T_contraction}
\mathcal T_{1+s}W(x)=\int_{\RR^d}p_{1+s}^1(x,x'')W(x'')dx''\le W(x)e^{-M_1 s}.
\end{equation}

\textbf{Step 2:} We will verify that we can apply Harris Ergodic Theorem. 

First, due to \eqref{T_contraction}, the function $W:X\to [0,\infty)$ and  $M_1>0$ satisfy 
\begin{equation}\label{prop1}
\int_{\RR^d}W(y)p_{1+s}^1(x,dy)\le e^{-M_1s} W(x), \quad \forall (s,x)\in(0,\infty)\times \RR^d.
\end{equation}

Second, recall that $|b_\infty(x)|\le\underset{|y|\le |x|}{\max}W(y) $ for all $y\in \RR^d$.
Then, for any fixed $R>0$  and any fixed $\hat s>0$,  Harnack's inequality  \cite[Theorem 8.2.1]{bogachev2015fokker} gives $\kappa(\hat s)>0$ such that
\begin{equation}\label{preprop2}
p_{1+\hat s}^1(x,x'')\ge \underset{|x|\le R}{\min}\,p_{1+\hat s/2}^1(x,0)\exp\big({-\kappa(\hat s)(1+(\underset{|y|\le |x''|}{\max}W(y))^2+|x''|^2)}\big),\quad \forall x''\in \RR^d, \,\, |x|\le R.
\end{equation}

Let $m(R):=\underset{|x|\le R}{\min}p_{1+\hat s/2}^1(x,0)$. Next we will show that $m(R)>0$. Indeed, following \cite[Proof of Lemma 3.5]{bogachev2019convergence}, let us consider $\psi\in C^{\infty}(\RR^d)$ with compact support such that $\psi(y)=1$ if $|y|\le 2R$ and $\psi(y)=0$ if $|y|>3R$ and let $|x|\le R$. For $t\ge 0$ we see that:
\[
\int_{\RR^d}\psi(y)p_{1+t}^1(x,dy)=\psi(x)+ \int_0^t\int_{\RR^d}L_\infty\psi(y)p_{1+s}^1(x,dy)ds.
\]
Consequently, 
\[
\int_{\RR^d}\psi(y)p_{1+t}^1(x,dy)\ge 1-t\underset{y}{\sup}|L_\infty\psi(y)|.
\]
Now by choosing $t$ small, we have that $\int_{\RR^d}\psi(y)p_{1+t}^1(x,dy)\ge 1/2$ for all $|x|\le R$. Since this holds for any such test function $\psi$, we conclude that $\underset{|y|\le 3R}{\min}p_{1+t}^1(x,y)\ge 1/2$ for all $|x|\le R$. And choosing $\hat s$ smaller if necessary, we can assume that $t=\hat s/4$ and applying Harnack's inequality again to prove that there exists $C>0$ such that  $1/2\le Cp_{1+\hat s/2}^1(x,0)$ or, in other words, $m(R)>0$ for every $R>0$.

Notice then that by \eqref{preprop2}, there exists $k\in (0,1)$ such that if $R>1/(1-e^{-M_1\hat s})$, then the probability measure $$\mu(dy)=\frac{1}{k}m(R)\exp\big({-\kappa(\hat s)(1+(\underset{|x''|\le 2|y|}{\max}W(x''))^2+|y|^2)}\big)dy$$ satisfies
\begin{equation}\label{prop2}
	\underset{x:W(x)\le R}{\inf}{p_{1+\hat s}^1(x, \cdot)}\ge k\mu.
\end{equation}

Then, by \eqref{prop1} and \eqref{prop2}, we can apply the Harris Ergodic Theorem \cite[Theorem 1.3]{hairer2011yet} and conclude the existence of two numbers, $\tilde\beta\in(0,1)$ and $\hat{\beta}>0$, such that for every two probability measures on $\RR^d$, $\mu_1,\mu_2$, 
\begin{equation}\label{contraction_hat}
\Big|\Big|\int_{\RR^d}p_{1+\hat s}^1(x,\cdot)\mu_1(dx)-\int_{\RR^d}p_{1+\hat s}^1(x,\cdot)\mu_2(dx)\Big|\Big|_{\hat{\beta}W-1}\le \tilde\beta||\mu_1-\mu_2||_{\hat{\beta}W-1}.
\end{equation}

Next, we apply repeatedly \eqref{contraction_hat} for two initial laws $\mathcal{L}(\xi),\mathcal{L}(\xi')\in \mathcal{P}^W(\RR^d)$ and  $(\mathcal T^* _{1+s})_{s\ge0}$. This allows us to conclude that for any $n\in \NN$ and fixed $T> 0$,  there exist  $\hat{\beta}>0,\tilde\beta\in(0,1)$ such that
\begin{equation*}
	||\mathcal T^*_{1+nT}(\LL(\xi))-\mathcal T^*_{1+nT}(\LL(\xi'))||_{\hat{\beta}W-1}\le \tilde\beta^n||\LL(\xi)-\LL(\xi')||_{\hat{\beta}W-1}.
\end{equation*}
In particular, since
\begin{align*}
	||\mathcal T^*_{1+nT}(\LL(\xi))-\mathcal T^*_{1+nT}(\LL(\xi'))||_{W}
	&\le2||\mathcal T^*_{1+nT}(\LL(\xi))-\mathcal T^*_{1+nT}(\LL(\xi'))||_{W-1}\\
	&=2\hat{\beta}^{-1}||\mathcal T^*_{1+nT}(\LL(\xi))-\mathcal T^*_{1+nT}(\LL(\xi'))||_{\hat{\beta}W-1}\le 2\hat{\beta}^{-1}\tilde\beta^n||\LL(\xi)-\LL(\xi')||_{\hat{\beta}W-1}
	\\
	&\le 2\tilde\beta^n||\LL(\xi)-\LL(\xi')||_{W},
\end{align*}
we conclude 
\begin{equation}\label{n_cte_first}
	||\mathcal T^*_{1+nT}(\LL(\xi))-\mathcal T^*_{1+nT}(\LL(\xi'))||_{W}\le  C\tilde\beta^n||\LL(\xi)-\mathcal \LL(\xi')||_W.
\end{equation}

Moreover, for any function $\psi\in C^{\infty}_0(\RR^d)$ such that $|\psi|\le CW$ and $0\le t< T$, we have $|\mathcal T_{1+t}\psi|\le \mathcal T_{1+t}|\psi|\le C\mathcal T_{1+t}W\le 2CW.$	Hence, 
\[
\int_{\RR^d}\psi(x'')(\mathcal T_{1+t}^*(\LL(\xi))-\mathcal T_{1+t}^*(\LL(\xi')))(dx'')=\int_{\RR^d}(\mathcal T_{1+t}\psi)(x'')(\LL(\xi)-\LL(\xi'))(dx'')\le 2C ||\LL(\xi)-\LL(\xi')||_W.
\]
Since $T>0$ was arbitrary, this and \eqref{n_cte_first} imply that there exist $C,c>0$ such that for all $s>0$,
\begin{equation}\label{n_cte}
	||\mathcal T^*_{1+s}(\LL(\xi))-\mathcal T^*_{1+s}(\LL(\xi'))||_{W}\le  Ce^{-cs}||\LL(\xi)-\mathcal \LL(\xi')||_W.
\end{equation}
Now,  as long as we have existence of an invariant measure $q$, applying this with $\xi'\sim q$ allows us to conclude \eqref{contraction Lyapunov_in_weighted}.

\textbf{Step 3:} Now, since  $(\mathcal{P}^W(\RR^d), ||\cdot||_{ W})$ is a complete metric space \cite[proof of Theorem 3.2]{hairer2011yet}, we can use the contraction \eqref{contraction Lyapunov_in_weighted} (for $s$ big enough such that $Ce^{-cs}<1$) to carry out a similar argument to that presented in Proposition \ref{W-moment_increment}. Namely we obtain, by Banach Fixed Point theorem, the existence of an invariant measure $q\in \mathcal{P}^W(\RR^d)$. 

\textbf{Step 4:} Finally,  we apply \eqref{contraction Lyapunov_in_weighted} with the initial condition $\xi=x\in \RR^d$ (a.s) and  the fact that any $\phi\in\mathcal{S}$ satisfies for some $C>0$ that $\phi\le C(1+W)$, in order to obtain for some $C,c>0$:
\begin{align*}
\underset{\phi \in \mathcal{S}_m}{\sup}\bigg|\int_{\RR^{d}}\phi(x') (p_{1+s}^1(x,dx')-q(dx'))\bigg|&\le	\underset{\phi \in \mathcal{S}_m}{\sup}\int_{\RR^{d}}|\phi(x')| \big|p_{1+s}^1(x,dx')-q(dx')\big|\le Cd_{W}(p_{1+s}^1(x,\cdot), q(\cdot))\\
&\le Ce^{-cs}d_{W}(\delta_x, q)= Ce^{-c s}\int_{\RR^{d}}( 1+ W(x-x''))q(dx'').
\end{align*}
Noticing that $W\ge 1$, we conclude from this the final statement.
\end{proof}
As a final remark, let us mention that there are even weaker alternative Lyapunov conditions allowing us to conclude convergence in WTV. An option is assuming that there exist a function $W(x): \RR^d\to [1,\infty)$, constants $M_1>0$, $M_2\in\RR$ and a compact set $K\subseteq \RR^d$ with index function $\textbf{1}_{K}$,  such that $\forall t\ge 0, \,x \in \RR^d,$
\[
L(t,x)W(x)\le M_2\textbf{1}_{K}-M_1W(x).
\]
Then we have exponential decay to an invariant measure in Weighted Total Variation distance (see \cite[Theorem 8.7]{bellet2006ergodic}) for test functions in $\mathcal{S}$ defined by \eqref{SLL}.

Next, we state the main theorem of this section: another example of time decaying derivative estimates for $V(t,x)=\EE[\phi(X_t^x)]$, where $X$ is solution to \eqref{original} and $\phi\in \mathcal S$ defined by \eqref{SLL}.
\begin{theorem}\label{thex1_for_good}

Suppose that Assumptions \ref{asex20}, \ref{asex2a}, \ref{coeff}  and \ref{asex3cc} hold.  Assume moreover that there exists $C_W>0$  such that  $\int_{\RR^d}|W(x)|e^{-c|x|^2}dx\le C_W$, where $c>0$ is any of the constants in  \eqref{f1},\eqref{f2}, \eqref{f3}. Then, there exists a measure $q\in \mathcal{P}^W(\RR^d)$ such that, for every $x\in \RR^d,s\ge 0$, it is a static measure for $Z_{1+s}^{1,x}$ solution to (\ref{limiting}). Moreover,  for $n=1,2$, $\phi \in \mathcal{S}$ defined in \eqref{SLL}, there exist $C,c>0$ such that:
\[
|\partial_x^nV(s,x)|\le C e^{-cs}\int_{\RR^d}\int_{\RR^d}W(x''-x')q(dx')e^{-c|x-x''|^{2}}dx'',\quad for\, all \,\, (s,x)\in[1,\infty)\times \RR^d.
\]
\end{theorem}
\begin{proof}
Recall that for any $x\in \RR^d$ the density of $(Z_{1+s}^{1,x})_{s\ge0}$ (solution to \eqref{limiting}) is denoted by $x'\mapsto p_{1+s}^1(x,x')$ and it has and invariant measure by $q$ by Proposition \ref{WTV_for good}. Let the density of $(X_{1+s}^{1,x})_{s\ge0}$ (solution to \eqref{auxintuition}) be denoted by $x'\mapsto P_{1+s}^1(x,x')$. 

Then, first by the Triangle inequality and afterwards by Lemma \ref{WTVmoments}, for any $x\in\RR^d,s\ge 0$
\begin{equation}\label{thisneedslabel}
\begin{split}
\int_{\RR^d}W(x')|P_{1+s}^1(x,dx')-p_{1+s}^1(x,dx')|&\le \int_{\RR^d}W(x') P_{1+s}^1(x,x')dx''+\int_{\RR^d}W(x') p_{1+s}^1(x,x')dx'\\
&=\EE[W(X_{1+s}^{1,x})]+\EE[W(Z_{1+s}^{1,x})]\\
&\le 2e^{-M_1 s}W(x).
\end{split}
\end{equation}
By the Triangle inequality, the fact that any $\phi\in\mathcal{S}$ satisfies for some $C>0$ that $\phi\le C(1+W)\le CW$, together with \eqref{thisneedslabel} and Proposition \ref{WTV_for good} , $\forall \,\, (s,x)\in[0,\infty)\times \RR^d$,
\begin{align*}
\underset{\phi\in \mathcal{S}}{\sup}&\bigg|\int_{\RR^{d}}\phi(x')(P_{1+s}^1(x,dx')-q(dx'))\bigg|\\
&\le\underset{\phi\in \mathcal{S}}{\sup}\bigg|\int_{\RR^{d}}\phi(x')(P_{1+s}^1(x,dx')-p_{1+s}^1(x,dx'))\bigg|+\underset{\phi\in \mathcal{S}}{\sup}\bigg|\int_{\RR^{d}}\phi(x')(p_{1+s}^1(x,dx')-q(dx')))\bigg|\\
&\le\int_{\RR^{d}}CW(x')|P_{1+s}^1(x,dx')-p_{1+s}^1(x,dx')|+\underset{\phi\in \mathcal{S}}{\sup}\bigg|\int_{\RR^{d}}\phi(x')(p_{1+s}^1(x,dx')-q(dx')))\bigg| \\
&\le  Ce^{-M_1 s}W(x)+Ce^{-c s}\int_{\RR^{d}} W(x-x'')q(dx'').
\end{align*}
In other words, Assumption \ref{as3} is satisfied  for some $C, c>0$ with $g(x)=W(x)+\int_{\RR^d}W(x-x'')q(dx'') $ and $G(s)=Ce^{-cs}$. 

Moreover, similarly to Lemma \ref{fried}, under Assumptions \ref{asex20} and \ref{asex2a} one asserts that Assumption \ref{as2} is satisfied with \\
$h(x)= C\int_{\RR^d}\int_{\RR^d}(W(x'')+W(x''-x'))q(dx')e^{-c|x-x''|^{2}}dx''$ since we assumed that $W$ is integrable against the specific Gaussians in the bounds \eqref{f1} and \eqref{f2}.
Alternatively, under Assumptions \ref{asex20} and  \ref{asex2b}(b), similarly to  Lemma \ref{eild} we can prove that  Assumption \ref{as2} is satisfied with $h(x)=  Ce^{(\nu+c)|x|^2}$.

Moreover, Assumption \ref{as4} is proved under both  Assumptions \ref{asex2a}(a) and  \ref{asex2a}(b) by similar arguments to those in Lemmas \ref{fried} and \ref{eild} respectively.

Finally, since Assumptions  \ref{as4}, \ref{as3} and \ref{as2} are satisfied, Theorem \ref{main} gives us the claim.
\end{proof}

\section{Connection to McKean-Vlasov SDEs}\label{mc}\label{ap_mc_sec}
Consider now the following real--valued, $d$-dimensional stochastic process $(X_{t})_{t\ge 0}$ satisfying a so called McKean-Vlasov SDE. That is, given coefficients $\beta:\RR^d\times \mathcal{P}_2(\RR^d)  \to \RR^d $ and $\Sigma:\RR^d\times \mathcal{P}_2(\RR^d) \to \RR^d \times \RR^d$ and any initial datum $\xi$ with $\mathcal{L}{(\xi)}\in \mathcal{P}_2(\RR^d)$,  our object of study is the stochastic process $X_s^{\xi}$ assumed to be the unique (in the sense of probability law) weak solution of the following SDE:
\begin{equation}\label{McKV1}
	dX_s^{\xi}=\beta(X_s^{\xi}, \mathcal{L}(X_s^{\xi}))ds +\Sigma(X_s^{\xi},\mathcal{L}(X_s^{\xi}))\,dB_s,\quad \forall s \in(0,\infty); \quad X_0^{\xi}=\xi.
\end{equation}

As usually, the McKean dynamics complicates the analysis. However, due to uniqueness, we can overcome these difficulties by the introduction of an auxiliary process. Assume  that indeed \eqref{McKV1} has a unique weak solution $\left(X^{\xi}_s\right)_{s\in [0,\infty)}$. We can now define the functions $b(s,x;\xi):=\beta(x,\mathcal{L}(X_s^{\xi}))$ and $\sigma(s,x;\xi):=\Sigma(x,\mathcal{L}(X_s^{\xi}))$ by just plugging in the law of the solution. Then, from \eqref{McKV1} we obtain the following related SDE:
\begin{equation}\label{original0}
	dY_s^{\tau,y;\xi}=b(s,Y_s^{\tau,y;\xi};\xi)ds+\sigma(s,Y_s^{\tau,y;\xi};\xi)dB_s;\, Y_{\tau}^{\tau,y;\xi}=y.
\end{equation}
Due to uniqueness and under monotonicity  conditions, one can prove that  for each $x\in \RR^d$,   $X_s^{x}=Y_s^{0,x;x}$ a.e, and so in particular  the functions $\V(s,x):=\EE[\phi(X_s^{0;x})]$ and $V(s,y;x):=\EE[\phi(Y_s^{0,y;x})]$, defined for $\phi$ a measurable function, satisfy $\V(s,x)=V(s,y;x)_{|_{y=x}}$ for all $x\in \RR^d$. Consequently, bounds on the derivatives of $\V$ can be obtained by exploiting uniquely the dynamic of $Y$, which  is that of a usual non-autonomous SDE, in the sense that $\partial_x\V(s,x)=\partial_{y}V(s,y;x)_{|_{y=x}}+\partial_{x}V(s,y;x)_{|_{y=x}}$. The reason why our main theorem can be applied to McKean--Vlasov SDEs is due to the fact that we can create this auxiliary non--autonomous process, where $b,\sigma$ preserve certain regularity and growth conditions. We can prove that for this type of processes we obtain either uniform estimates or decay in time for the derivatives $[1,\infty)\times \RR^d \ni(s,x)\mapsto \partial_{y}^nV(s,y;x)_{|_{y=x}}$ of orders $n$ dictated by an assumption on the regularity and integrability of the density of $Y_1^{0,x;x}$. Again, uniform estimates can be concluded when the transition densities of the process $Y$ remain ``close'' to some fixed measure $q\in \mathcal{P}_2(\RR^d)$ and only when we succeed on proving decay towards such a measure  we achieve decay of the derivatives as well.

Going back to the process solving equation \eqref{McKV1}, it arises naturally as limit of a weakly interacting particle system (we denote it's elements by $Z^{i,N}$, where $N$ is the number of particles in the system). It turns out that a particular limiting behaviour of this system is that as $N$ is allowed to go to infinity, any finite subset of particles become asymptotically independent of each other and in fact they converge weakly to i.i.d copies of the original McKean--Vlasov process  $X_s^{\xi}$. 
Now, there is a version of this phenomenon known by the name of weak propagation of chaos, which deals with the statistical behaviour of the empirical distribution of the particle system and whose explicit bounds are concerned with estimates of expressions of the form $|\EE[\phi(X_T)]-\EE[\phi(Z_T^{1,N})]|$ for some measurable test function $\phi\in B(\RR^d)$ and some fixed time $T>0.$ Recent work presented in \cite{chassagneux2019weak} and \cite{bencheikh2019weak} asserts order 
$\mathcal{O}\left(\frac{1}{N}\right)$ for this kind of expressions and order 
$\mathcal{O}\left(\frac{1}{N}+h\right)$ for the step--$h$ Euler discretization scheme. Now, uniform estimates in time are harder to come by and have only recently been proved for the torus in
\cite{delarue2021uniform} when the diffusion is constant and the drift belongs to some specific family such as the ones with small dependence on the measure. 
As opposed to the Master Equation used in \cite{delarue2021uniform}, for some  cases it is enough to use regular PDE techniques for obtaining uniform estimates for the weak propagation of chaos. They are based on the uniform decay of the first and second space derivatives of the real solution (in variables $(s,y)$) to the backward Kolmogorov equation, i.e  $V: [0,\infty) \times \RR^d\times \RR^d\ni (s,y,x) \mapsto \EE[\phi(Y_s^{0,y;x})]$, that can be obtained using our approach. 


\section*{Acknowledgments}
Maria Lefter was supported by The
Maxwell Institute Graduate School in Analysis and its Applications, a Centre for Doctoral Training funded by the
UK Engineering and Physical Sciences Research Council (grant EP/L016508/01), the Scottish Funding Council,
Heriot-Watt University, and the University of Edinburgh.

\begin{appendix}
\section{Chapman--Kolmogorov identity for non-homogeneous processes}\label{CK}
Consider a stochastic process defined for $s\ge 0$  by:
\begin{equation}\label{processA}
X_{\tau+s}^{\tau,x}=x+\int_{\tau}^{\tau+s}b(r,X_r^{\tau,x} )dr +\int_{\tau}^{\tau+s}\sigma(r,X_r^{\tau,x} )dB_r,
\end{equation}
and which has a density $[\tau,\infty)\times \RR^d \ni (\tau+s,x')\to p_{\tau+s}^{\tau}(x,x')$. Then, for any function $\phi\in C^0(\RR^d)$  with compact support, we have that 
$$\EE[\phi(X_{\tau+s}^{\tau,x})]=\int_{\RR^d}\phi(x')p_{\tau+s}^{\tau}(x,x')dx'.$$

\begin{lemma}\label{ckappendix1}
Given $(\tau,x)\in [0,\infty)\times \RR^d$, assume that the process $X=(X_{\tau+s}^{\tau,x})_{s\ge0}$ defined by \eqref{processA} satisfies the flow property. Then, its density $[\tau,\infty)\times \RR^d \ni (\tau+s,x')\to p_{\tau+s}^{\tau}(x,x')$ satisfies the following for any $\tau\le 1$, $s\ge 0$ and $x,x''\in \RR^d$:
\begin{equation}\label{ckappendix}p_{\tau+s}^{\tau}(x,x'')=\int p_{\tau+s}^{1}(x',x'')p_{1}^{\tau}(x,x')dx'.
\end{equation}
\end{lemma}

\begin{proof}
For $\tau\le 1$ recall that we assumed   for any $\tau, s, x,$
$\displaystyle{X_{\tau+s}^{\tau,x}=X_{\tau+s}^{{1+\tau},X_{1+\tau}^{\tau,x}}}.$

First by the tower property and afterwards by the  flow property above, for any  $\phi\in C^0(\RR^d)$ with compact support we have:
\begin{align*}
\int_{\RR^d}\phi(x')p_{\tau+s}^{\tau}(x,x')dx'&=\EE[\phi(X_{\tau+s}^{\tau,x})]=\EE\big[\EE[\phi(X_{\tau+s}^{{1},X_{1}^{\tau,x}})]|X_{1}^{\tau,x}\big]=\EE\big[\EE[\phi(X_{\tau+s}^{{1},X_{1}^{\tau,x}})]|X_{1}^{\tau,x}\big]\\
&=\int_{\RR^d}\EE[\phi(X_{\tau+s}^{{1},x''})]p_{1}^{\tau}(x,x'')dx''=\int_{\RR^d}\int_{\RR^d}\phi(x')p_{\tau+s}^{{1}}(x'',x')p_{1}^{\tau}(x,x'')dx'dx''\\
&=\int_{\RR^d}\phi(x')\bigg(\int_{\RR^d} p_{\tau+s}^{{1}}(x'',x')p_{1}^{\tau}(x,x'')dx''\bigg)dx'.
\end{align*}	
Finally, since the above holds for an arbitrary test function $\phi,$ we conclude \eqref{ckappendix}.
\end{proof}	
\end{appendix}

\bibliographystyle{ieeetr}
\bibliography{Non-autonomous-SDEs.bib}	

\end{document}